\documentclass[a4paper,reqno,10pt]{amsart}
\usepackage{a4wide}
\usepackage{amsmath,amssymb,amsfonts,amsthm,mathtools}
\usepackage{amscd}
\usepackage{tikz-cd}
\usetikzlibrary{arrows.meta,positioning}
\usepackage{enumitem}
\usepackage{microtype}
\usepackage[colorlinks=true,linkcolor=blue,citecolor=blue,urlcolor=blue]{hyperref}
\setlist[enumerate]{leftmargin=*}
\setlist[enumerate,1]{label={\upshape(\arabic*)}}
\setlist[enumerate,2]{label={\upshape(\alph*)}}
\newtheorem{theorem}{Theorem}[section]
\newtheorem{proposition}[theorem]{Proposition}
\newtheorem{lemma}[theorem]{Lemma}
\newtheorem{corollary}[theorem]{Corollary}
\newtheorem*{theoremA}{Theorem A}
\newtheorem*{theoremB}{Theorem B}
\theoremstyle{definition}
\newtheorem{definition}[theorem]{Definition}
\newtheorem{setting}[theorem]{Setting}
\newtheorem{construction}[theorem]{Construction}
\newtheorem{example}[theorem]{Example}
\newtheorem{remark}[theorem]{Remark}
\newtheorem*{organization}{Organization}
\newtheorem*{conventions}{Conventions and notation}
\newtheorem*{formalization}{Formalization}
\newtheorem*{aiuse}{Use of AI}
\numberwithin{equation}{section}
\DeclareMathOperator{\moduleCategory}{\mathsf{mod}}
\renewcommand{\mod}{\moduleCategory}
\DeclareMathOperator{\ind}{ind}
\DeclareMathOperator{\add}{add}
\DeclareMathOperator{\Hom}{Hom}
\DeclareMathOperator{\Ext}{Ext}
\DeclareMathOperator{\End}{End}
\DeclareMathOperator{\Irr}{Irr}
\DeclareMathOperator{\rad}{rad}
\DeclareMathOperator{\soc}{soc}

\newcommand{\op}{\mathrm{op}}
\newcommand{\cC}{\mathcal C}
\newcommand{\cD}{\mathcal D}

\newcommand{\one}{\mathbf 1}
\newcommand{\indicator}{\mathbf 1}

\newcommand{\TSp}{T\text{-}\mathsf{Sp}}
\newcommand{\Dk}{D_k}
\newcommand{\indeg}{\operatorname{indeg}}
\newcommand{\outdeg}{\operatorname{outdeg}}
\newcommand{\uhom}{\underline{\Hom}}
\newcommand{\ohom}{\overline{\Hom}}

\title{Magnitude of module categories and special biserial algebras}
\author[H.\ Enomoto]{Haruhisa Enomoto}
\address{Parakeet Inc., Japan}
\email{haruhisa.enomoto.math@gmail.com}
\date{\today}
\begin{document}
\begin{abstract}
B{\o}rve, Horiatakis, and Kalck defined the magnitude of the module
category of a representation-finite algebra as the sum of the entries of
the inverse of the matrix of Hom dimensions between indecomposable
modules. They showed it is an Euler characteristic, the alternating
count of vertices, arrows, and meshes of the Auslander--Reiten quiver,
and conjectured that it is at least the number of simple modules, with
equality precisely for special biserial algebras. We prove this
conjecture over an algebraically closed field. For a
representation-directed algebra, we relate the magnitudes of the module
category, the subcategory of modules without a given simple composition
factor, and the ideal quotient by it. By the theory of
hammocks of Ringel and Vossieck, this quotient is equivalent to a
category of representations of a finite poset, whose magnitude we
compute. The general case reduces to the representation-directed case through algebras obtained
by grading a standard form.
\end{abstract}
\maketitle
\setcounter{tocdepth}{1}
\tableofcontents

\section{Introduction}\label{sec:introduction}
Magnitude is an invariant of enriched categories introduced by Leinster
\cite[Section~1.3]{Lei13}; for finite categories it generalizes the Euler characteristic
of the classifying space \cite{Lei08}. Chuang, King, and Leinster studied the
magnitude of finite-dimensional algebras \cite{CKL}. B{\o}rve, Horiatakis,
and Kalck recently defined the magnitude of the module category of a
representation-finite algebra \(A\) as the magnitude of its Auslander
algebra \cite{BHK}. Over an algebraically closed field \(k\), it is
\[
 \chi(\mod A)=\one^{\mathsf T}H_A^{-1}\one,\qquad
 H_A=\bigl(\dim_k\Hom_A(X,Y)\bigr)_{X,Y\in\ind A},
\]
the sum of the entries of the inverse of the invertible Hom matrix between the
indecomposable modules. They showed that it is the Euler characteristic of
the Auslander--Reiten quiver \(\Gamma_A\): the alternating sum of the
numbers of vertices, arrows, and meshes, which correspond to cells of
dimensions \(0\), \(1\), and \(2\) \cite[Lemma~1.9]{BHK}:
\begin{equation}\label{eq:intro-count}
 \chi(\mod A)=|\ind A|-a(A)+\bigl(|\ind A|-|A|\bigr).
\end{equation}
Here \(a(A)\) is the number of arrows of \(\Gamma_A\), and \(|A|\) is the
number of simple \(A\)-modules up to isomorphism, so that
\(|\ind A|-|A|\) counts the meshes, one for each nonprojective
indecomposable.

B{\o}rve, Horiatakis, and Kalck computed this invariant for several classes
of algebras. It equals \(|A|\) for representation-finite special biserial
algebras \cite[Theorem~3.5]{BHK}, and \(h_Q-1\) for the path algebra of a
Dynkin quiver \(Q\) with Coxeter number \(h_Q\) \cite[Theorem~4.3]{BHK}.
They conjectured that \(\chi(\mod A)\geq|A|\) for every
representation-finite bound path algebra over a field, with equality
precisely for special biserial algebras \cite[Conjecture~A]{BHK}. They
proved this for representation-finite hereditary algebras (a consequence
of \cite[Theorem~4.3]{BHK}), radical square zero algebras
\cite[Corollary~4.6]{BHK}, and self-injective algebras
\cite[Theorem~5.8 and Corollary~5.9]{BHK}. We prove the conjecture over an
algebraically closed field.

\begin{theoremA}[= Theorem~\ref{thm:main}]\label{thm:intro-main}
Let \(A\) be a finite-dimensional representation-finite algebra over an
algebraically closed field. Then the following hold.
\begin{enumerate}
\item \(\chi(\mod A)\geq |A|\).
\item \(\chi(\mod A)=|A|\) if and only if \(A\) is special biserial.
\end{enumerate}
\end{theoremA}

By \eqref{eq:intro-count}, Theorem~A is a statement about the shape of
Auslander--Reiten quivers: the number of arrows is at most twice the
number of nonprojective indecomposables, with equality precisely for
special biserial algebras. We call the difference
\[
 \sigma(A)=\chi(\mod A)-|A|=2\bigl(|\ind A|-|A|\bigr)-a(A)
\]
the \emph{magnitude surplus}.

\medskip\noindent\textbf{Strategy of the proof.}
The proof first treats representation-directed algebras by removing one
simple module at a time. For a primitive idempotent \(e\) of \(A\), we
call the algebra \(A/AeA\) the \emph{primitive deletion} of \(e\). For a basic, representation-finite, and representation-directed
algebra \(A\), we show that
\[
 \sigma(A/AeA)\leq\sigma(A),
\]
with equality only if \(\dim_k Xe\leq1\) for every indecomposable
\(A\)-module \(X\) (Proposition~\ref{thm:directed-deletion}).

This comes from a general identity (Theorem~\ref{thm:serre-magnitude}).
For an idempotent \(e\) of \(A\), put \(B=A/AeA\), and let
\(\cD=(\mod A)/[\mod B]\) be the \emph{ideal quotient} by the maps
factoring through \(\mod B\). Like \(\chi(\mod A)\), the magnitude of \(\cD\) is
the Euler characteristic of its translation quiver \(\Gamma_{\cD}\)
(Proposition~\ref{thm:categorical-magnitude} and
Definition~\ref{def:boundary}), and
\[
 \chi(\mod A)-\chi(\mod B)-\chi(\cD)=c-r-\delta,
\]
where \(c\) is the sum, over pairs of indecomposable \(B\)-modules, of
the difference between the numbers of arrows between them in \(\Gamma_B\)
and in \(\Gamma_A\), \(r\) is the number of almost split sequences of
\(\mod B\) that are not almost split in \(\mod A\), and \(\delta\geq0\)
counts, for each almost split sequence of \(\mod A\) with exactly one end
term in \(\mod B\), all but one of its middle summands outside
\(\mod B\). When \(A\) is representation-finite
and representation-directed and \(e\) is primitive, \(\delta=0\) and
\(r\leq c\); since \(|A|=|B|+1\), it then suffices to show
\(\chi(\cD)\geq1\). In this case
Ringel and Vossieck identified \(\cD\) with the category \(\TSp\) of
representations of a finite poset \(T\) by subspaces
\cite[Section~10, Theorem, p.~243]{RV}: such a representation is a vector
space \(W\), its \emph{total space}, with subspaces \(W_t\) for \(t\in T\)
such that \(W_s\subseteq W_t\) whenever \(s\leq t\). We compute the magnitude of these
categories.

\begin{theoremB}[= Theorem~\ref{thm:poset-magnitude}]
Let \(T\) be a finite poset such that \(\TSp\) is representation-finite.
In the Auslander--Reiten quiver of \(\TSp\), all paths from the
representation \(k\) with all subspaces zero to the representation \(k\)
with all subspaces \(k\) have the same length \(L\), and
\[
 \chi(\TSp)=1+L-|T|\geq1.
\]
Equality holds if and only if \(T\) has width at most two, if and only if
every indecomposable has one-dimensional total space.
\end{theoremB}

The inequality \(L\geq|T|\) and the equivalence of \(L=|T|\) with the two
conditions of Theorem~B are due to Ringel and Vossieck
\cite[Section~8, p.~237]{RV}; the new part of Theorem~B is the formula for
\(\chi(\TSp)\). For a basic, representation-finite, and
representation-directed algebra \(A\),
deleting simple modules one at a time then gives \(\sigma(A)\geq0\). If
\(\sigma(A)=0\), then \(\sigma(A/AeA)=\sigma(A)\) for every primitive
idempotent \(e\), so \(\dim_k Xe\leq1\) for all \(e\) and all
indecomposable \(X\). Thus every indecomposable \(A\)-module
is thin, that is, has each simple module at most once as a composition
factor. Hence \(A\) is biserial by
Proposition~\ref{prop:thin-biserial}, due to Pogorza{\l}y and
Skowro\'nski, and special biserial by Theorem~\ref{thm:biserial-special},
due to Skowro\'nski and Waschb\"usch.

For a general \(A\), we pass to a standard form \(S\) of \(A\)
\cite{BG,BrG}: an algebra with the same Auslander--Reiten quiver, whose
indecomposable modules and maps between them are given by paths in this
quiver modulo the mesh relations. Path length grades the Hom spaces
between indecomposable \(S\)-modules, and hence \(S\). The graded \(S\)-modules in degrees
\(0,\ldots,m\) are the modules over a representation-directed algebra
\(S_m\), which we call an interval algebra. One has
\(\sigma(S_m)=(m+1)\sigma(A)+O(1)\) as \(m\to\infty\)
(Theorem~\ref{prop:interval-count}, with \(\sigma(S)=\sigma(A)\)), and
\(\sigma(S_m)\geq0\) by Corollary~\ref{cor:directed}; hence \(\sigma(A)\geq0\)
(Section~\ref{sec:general}).

Now suppose \(\sigma(A)=0\). The interval algebra \(S_M\) is block upper
triangular, with one diagonal block for each degree in \([0,M]\). For
every \(q\geq1\) and suitable \(M\), deleting the degrees in the gaps
between \(q\) well separated copies of \([0,m]\), that is, factoring out
their diagonal idempotents, leaves a product of \(q\) copies of \(S_m\). The surplus is additive on products and does not
increase under deletion, so \(q\,\sigma(S_m)\leq\sigma(S_M)\), and
\(\sigma(S_M)=O(1)\) as \(q\) grows because \(\sigma(A)=0\).
Hence \(\sigma(S_m)=0\) for every \(m\)
(Lemma~\ref{lem:separated-intervals}), and every \(S_m\) is biserial by
Corollary~\ref{cor:directed}. By results of Auslander and Reiten and of
Skowro\'nski and Waschb\"usch, a representation-finite algebra is
biserial if and only if the middle term of every almost split sequence
has at most two nonprojective indecomposable summands
(Theorem~\ref{thm:biserial-ar}).
The almost split sequences of \(\mod S\) lift to almost split sequences
over \(S_m\) for large \(m\), with the same nonprojective middle
summands, and \(A\) and \(S\) have isomorphic Auslander--Reiten quivers.
Hence \(A\) is biserial, and special biserial by
Theorem~\ref{thm:biserial-special} (Section~\ref{sec:equality}).
Figure~\ref{fig:proof-overview} summarizes these steps.

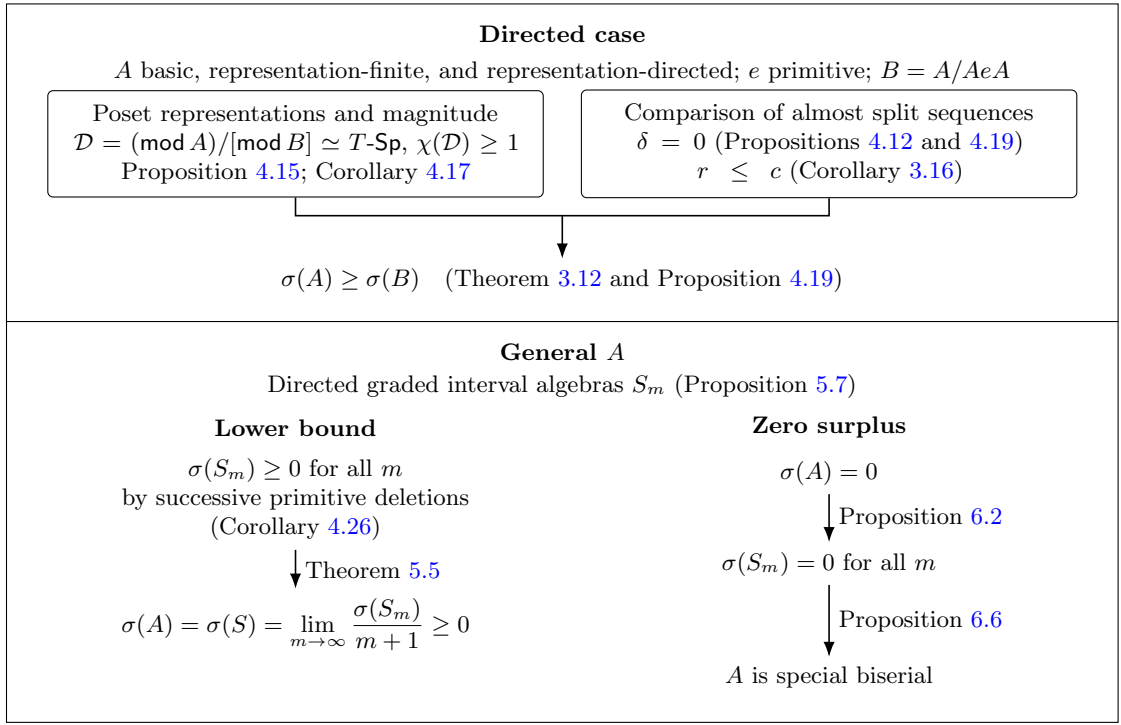
\begin{figure}[htbp]
\centering
\begin{tikzpicture}[
  x=0.98cm, font=\small,
  every node/.style={align=center},
  input/.style={draw, rounded corners=2pt, text width=6.2cm,
    minimum height=1.25cm, inner sep=5pt},
  implication/.style={-{Latex}, semithick}]
\draw (0,0) rectangle (15,-9.5);
\node[font=\small\bfseries] at (7.5,-0.4) {Directed case};
\node at (7.5,-0.9)
  {\(A\) basic, representation-finite, and representation-directed;
   \(e\) primitive; \(B=A/AeA\)};
\node[input] (poset) at (3.9,-1.85)
  {Poset representations and magnitude\\
   \(\cD=(\mod A)/[\mod B]\simeq\TSp\), \(\chi(\cD)\geq1\)\\
   Proposition~\ref{prop:poset-equivalence}; Corollary~\ref{prop:intrinsic}};
\node[input] (sequences) at (11.1,-1.85)
  {Comparison of almost split sequences\\
   \(\delta=0\) (Propositions~\ref{prop:boundary-one} and~\ref{thm:directed-deletion})\\
   \(r\leq c\) (Corollary~\ref{cor:directed-compensation})};
\node (deletion) at (7.5,-3.65)
  {\(\sigma(A)\geq\sigma(B)\)\quad (Theorem~\ref{thm:serre-magnitude}
   and Proposition~\ref{thm:directed-deletion})};
\draw[semithick] (poset.south) -- (3.9,-2.8) -- (11.1,-2.8)
  -- (sequences.south);
\draw[implication] (7.5,-2.8) -- (deletion.north);
\draw (0,-4.2) -- (15,-4.2);
\node[font=\small\bfseries] at (7.5,-4.6) {General \(A\)};
\node at (7.5,-5.05)
  {Directed graded interval algebras \(S_m\)
   (Proposition~\ref{prop:graded-classification})};
\node[font=\small\bfseries] at (3.9,-5.6) {Lower bound};
\node[font=\small\bfseries] at (11.1,-5.6)
  {Zero surplus};
\node (nonnegative) at (3.9,-6.55)
  {\(\sigma(S_m)\geq0\) for all \(m\)\\
   by successive primitive deletions\\
   (Corollary~\ref{cor:directed})};
\node (limit) at (3.9,-8.2)
  {\(\displaystyle\sigma(A)=\sigma(S)=\lim_{m\to\infty}
     \frac{\sigma(S_m)}{m+1}\geq0\)};
\draw[implication] (nonnegative.south) --
  node[right] {Theorem~\ref{prop:interval-count}} (limit.north);
\node (equality) at (11.1,-6.2) {\(\sigma(A)=0\)};
\node (intervals) at (11.1,-7.4)
  {\(\sigma(S_m)=0\) for all \(m\)};
\node (biserial) at (11.1,-8.9) {\(A\) is special biserial};
\draw[implication] (equality.south) --
  node[right] {Proposition~\ref{prop:interval-equality}} (intervals.north);
\draw[implication] (intervals.south) --
  node[right] {Proposition~\ref{prop:equality-forward}} (biserial.north);
\end{tikzpicture}
\caption{Primitive deletion and the two parts of Theorem~A.
 Here \(\sigma(A)=\chi(\mod A)-|A|\), \(\TSp\) denotes
 representations of a finite poset \(T\) by subspaces, \(S\) is a
 standard form of \(A\) with \(\sigma(S)=\sigma(A)\)
 (Theorem~\ref{thm:standard} and Lemma~\ref{lem:standard-grading}), and \(\mod S_m\) is equivalent to the category
 of graded \(S\)-modules supported in degrees \(0,\ldots,m\)
 (Proposition~\ref{prop:interval-modules}).}
\label{fig:proof-overview}
\end{figure}

\begin{organization}
Section~\ref{sec:ar} recalls almost split sequences and Iyama's
\(\tau\)-categories, and extends the proof of \eqref{eq:intro-count} in
\cite{BHK} to strict \(\tau\)-categories. Section~\ref{sec:quotient} proves the
identity of Theorem~\ref{thm:serre-magnitude} for the ideal quotient by a
Serre subcategory and a
sufficient condition for the inequality \(r\leq c\). Section~\ref{sec:deletion} proves
Theorem~B and the inequality \(\sigma(A/AeA)\leq\sigma(A)\) for a primitive
idempotent \(e\). Section~\ref{sec:general} defines interval algebras
and proves Theorem~A(1), and Section~\ref{sec:equality} proves
Theorem~A(2).
Appendix~\ref{app:structural} gives direct proofs of the three results on
the ideal quotient in Section~\ref{sec:deletion} that use \cite{RV}, so that Theorem~A does
not depend on \cite{RV}; Theorem~B for an arbitrary poset does.
\end{organization}

\begin{conventions}
Unless specified otherwise, \(k\) is algebraically closed and algebras are
finite-dimensional \(k\)-algebras. Modules are finite-dimensional right modules, and
\(\mod A\) denotes their category. Composition is written
\(gf\) for first \(f\), then \(g\). We use the same order for paths:
\(\beta\alpha\) is the path \(\alpha\) followed by \(\beta\). For a finite
quiver \(Q\), right modules over \(A=(kQ)^{\op}\), or over its quotient by
a specified ideal, are identified
with representations of \(Q\): the space at a
vertex \(i\) is \(Me_i\), and an arrow \(\alpha:i\to j\) acts as a map
\(Me_i\to Me_j\). We write \(P_i\), \(I_i\), and \(S_i\) for the
indecomposable projective, injective, and simple modules at the vertex
\(i\). For a module \(M\), \(|M|\) is the number
of isomorphism classes of indecomposable direct summands of \(M\); thus
\(|A|=|A_A|\) is the number of isomorphism classes of simple modules.
We write \(\Dk=\Hom_k(-,k)\), and
\(\add M\) for direct summands of finite direct sums of copies of \(M\).
For a Krull--Schmidt category \(\cC\), \(\ind\cC\) is a chosen set
of indecomposable representatives, also regarded as the full subcategory
on those objects. We write \(\cC(X,Y)=\Hom_{\cC}(X,Y)\) and abbreviate
\(\ind(\mod A)\) to \(\ind A\). A category is representation-finite
if it has finitely many indecomposable isomorphism classes.
The categorical radical is \(\rad_{\cC}\), and
\(\Irr_{\cC}(X,Y)=\rad_{\cC}(X,Y)/\rad_{\cC}^2(X,Y)\).
The \emph{irreducible quiver} of \(\cC\) has vertices \(\ind\cC\)
and \(\dim_k\Irr_{\cC}(X,Y)\) arrows from \(X\) to \(Y\).
We write \(\indeg_{\cC}(X)\) and \(\outdeg_{\cC}(X)\) for the
numbers of incoming and outgoing arrows, including multiplicities.
The \emph{Auslander--Reiten quiver} \(\Gamma_A\) is this
quiver of \(\mod A\), equipped with Auslander--Reiten translation and
with its projective and injective vertices marked.
Translations are displayed by dotted arrows
\(X\dashrightarrow\tau_A X\); they do not count as arrows in paths
or arrow counts.
An algebra is \emph{representation-directed} if there is no cycle of
nonzero nonisomorphisms between indecomposable modules. We also use
\emph{directed} in this sense for an algebra and for a category.
Special biserial and biserial algebras are recalled in
Definitions~\ref{def:special-biserial} and~\ref{def:biserial}; a non-basic
algebra is special biserial if its basic algebra is.
\end{conventions}

\begin{formalization}
A Lean~4 formalization of Theorem~A is available at
\url{https://github.com/haruhisa-enomoto/magnitude-conjecture}.
It follows the proof in this paper, with the arguments of
Appendix~\ref{app:structural}.
\end{formalization}

\begin{aiuse}
The research was done under the author's direction using GPT-5.6-Sol,
Claude Fable~5, Claude Opus~5, and GPT-6-Astra. GPT-5.6-Sol and Claude Fable~5 developed
the initial proof, and GPT-5.6-Sol produced earlier manuscript drafts.
GPT-6-Astra subsequently developed the simplified proof presented here and
wrote a new manuscript based on it, keeping selected arguments and examples
from the earlier drafts. It also produced the Lean formalization.
Claude Opus~5 extensively reviewed the proofs and
successive manuscript drafts. The author directed the research and revisions and
set the structure and much of the wording of the present text. The author
is responsible for the mathematical content and the final manuscript.
\end{aiuse}

\section{Preliminaries on \texorpdfstring{\(\tau\)}{tau}-categories and magnitude}\label{sec:ar}
We recall almost split sequences and their generalization to additive
categories. Each such sequence gives the beginning of a projective
resolution of a simple functor. Following the proof of
\eqref{eq:intro-count} in \cite[Lemma~1.9]{BHK}, these resolutions give an
analogue of \eqref{eq:intro-count} (Proposition~\ref{thm:categorical-magnitude})
for strict \(\tau\)-categories (Definition~\ref{def:strict-tau}), in which
the resolutions have length two. Module categories and the ideal quotients
used below are strict \(\tau\)-categories.

\subsection{Almost split maps and \texorpdfstring{\(\tau\)}{tau}-sequences}\label{subsec:tau-seq}
We recall almost split maps and sequences, then define the
\(\tau\)-sequences that replace them in ideal quotients.

Let \(\cC\) be a Krull--Schmidt \(k\)-linear category, for example
\(\mod A\). A map \(f:V\to X\) in \(\cC\) to an indecomposable \(X\) is
\emph{right almost split}
if it is not a retraction and every non-retraction to \(X\) factors through it.
It is \emph{right minimal} if every \(u\in\End(V)\) with \(fu=f\)
is invertible. Left almost split maps and left minimality are defined dually,
using sections.

\begin{definition}\label{def:ar}
Let \(A\) be an algebra. For a nonprojective indecomposable \(A\)-module
\(X\), write its almost split sequence as
\[
 0\longrightarrow\tau_A X\xrightarrow{\nu_X^+}\theta_A X
 \xrightarrow{\mu_X^+}X\longrightarrow0.
\]
At a projective \(X\), put \(\tau_A X=0\) and \(\theta_A X=\rad X\), and
let \(\mu_X^+\) be the inclusion.
Dually, write the sequence starting at a noninjective \(X\) as
\[
 0\longrightarrow X\xrightarrow{\mu_X^-}\theta_A^-X
 \xrightarrow{\nu_X^-}\tau_A^-X\longrightarrow0.
\]
At an injective \(X\), put \(\tau_A^-X=0\) and \(\theta_A^-X=X/\soc X\),
and let \(\mu_X^-\) be the quotient map.
Extend \(\tau_A,\tau_A^-,\theta_A,\theta_A^-\) additively to arbitrary
modules, as assignments on isomorphism classes.
\end{definition}

The following standard facts relate almost split maps to the arrows of
\(\Gamma_A\). For an indecomposable \(Y\), write \([V:Y]_\oplus\)
for its multiplicity as a direct summand of \(V\).

\begin{proposition}[{\cite[V, Section~1; VII, Proposition~1.3]{ARS}}]\label{prop:ar-multiplicity}
Let \(X,Y\) be indecomposable \(A\)-modules. Then \(\mu_X^+\) is minimal
right almost split and \(\mu_X^-\) is minimal left almost split, and
\begin{equation}\label{eq:ar-multiplicity}
 [\theta_A X:Y]_\oplus=\dim_k\Irr_A(Y,X),\qquad
 [\theta_A^-X:Y]_\oplus=\dim_k\Irr_A(X,Y).
\end{equation}
The map \(\tau_A\) is a bijection from nonprojective to noninjective
indecomposables, with inverse \(\tau_A^-\).
\end{proposition}

We will compute Ext groups from Hom groups by the following form of
Auslander--Reiten duality.
\begin{proposition}[{\cite[IV, Theorem~2.13]{ASS}}]\label{prop:ar-duality}
For modules \(M,N\), there are natural isomorphisms
\begin{equation}\label{eq:ar-duality}
 \Dk\Ext_A^1(M,N)\cong\uhom_A(\tau_A^-N,M)
 \cong\ohom_A(N,\tau_A M).
\end{equation}
Here \(\uhom\) is Hom modulo maps through projectives, and \(\ohom\)
is Hom modulo maps through injectives. The formulas include the zero cases
when \(N\) is injective or \(M\) is projective.
\end{proposition}

The ideal quotients studied in Section~\ref{sec:quotient} are in general
not abelian. In them, almost split sequences are replaced by Iyama's
\(\tau\)-sequences, which we define through functors on the category.
\begin{definition}\label{def:simple-functors}
Let \(\cC\) be a Hom-finite Krull--Schmidt \(k\)-linear category. Write
\(\mod\cC\) for the abelian category of contravariant \(k\)-linear functors
from \(\cC\) to finite-dimensional vector spaces. For an object \(X\), put
\(\mathsf P_X=\cC(-,X)\), a projective object of \(\mod\cC\) by the Yoneda
lemma, and for a map \(f:L\to M\) write
\(\mathsf P_f=f\circ-:\mathsf P_L\to\mathsf P_M\). For an indecomposable
\(X\), put
\[
 \mathsf S_X=\mathsf P_X/\rad_{\cC}(-,X).
\]
It is simple, with \(\mathsf S_X(Y)=0\) for indecomposables \(Y\not\cong X\),
and \(\mathsf S_X(X)=\End_{\cC}(X)/\rad_{\cC}(X,X)=k\), since this is a
finite-dimensional division algebra over the algebraically closed field
\(k\).
Dually, for covariant functors put \(\mathsf P^X=\cC(X,-)\),
\(\mathsf P^f=-\circ f:\mathsf P^M\to\mathsf P^L\), and
\(\mathsf S^X=\mathsf P^X/\rad_{\cC}(X,-)\).
\end{definition}

\begin{definition}\label{def:weak-kernel}
A map \(f:L\to M\) in \(\cC\) is a \emph{weak kernel} of \(g:M\to N\) if
\(\mathsf P_L\xrightarrow{\mathsf P_f}\mathsf P_M\xrightarrow{\mathsf P_g}\mathsf P_N\)
is exact, that is, \(gf=0\) and every map \(h:W\to M\) with \(gh=0\)
factors through \(f\). It is minimal if \(fu=f\) forces \(u\in\End(L)\)
to be invertible. Weak cokernels and their minimality are defined dually,
using \(\mathsf P^N\to\mathsf P^M\to\mathsf P^L\).
\end{definition}

\begin{definition}[{\cite[Section~2.1]{IyamaI}}; in the formulation of {\cite[Section~1.3]{IyamaIV}}]\label{def:tau-category}
Let \(\cC\) be as in Definition~\ref{def:simple-functors}. A complex
\(L\xrightarrow{f}M\xrightarrow{g}N\) with \(f,g\) radical is a
\emph{right \(\tau\)-sequence} if \(f\) is a minimal weak kernel of \(g\)
and the sequences
\[
 \mathsf P_M\xrightarrow{\mathsf P_g}\rad_{\cC}(-,N)\longrightarrow0,
 \qquad
 \mathsf P^M\xrightarrow{\mathsf P^f}\rad_{\cC}(L,-)\longrightarrow0
\]
are exact. A \emph{left \(\tau\)-sequence} satisfies the same conditions,
with \(g\) a minimal weak cokernel of \(f\) instead. The category \(\cC\)
is a \emph{\(\tau\)-category} if every indecomposable is the last term of
a right \(\tau\)-sequence and the first term of a left
\(\tau\)-sequence.
\end{definition}
The terms \(L\) and \(M\) of a right \(\tau\)-sequence may be zero, and
dually for left \(\tau\)-sequences.
The two exactness conditions say that every radical map into \(N\)
factors through \(g\), and every radical map out of \(L\) factors through
\(f\). For a right \(\tau\)-sequence ending at an indecomposable \(N\), the
first condition and the weak kernel property give an exact sequence
\begin{equation}\label{eq:right-presentation}
 \mathsf P_L\xrightarrow{\mathsf P_f}\mathsf P_M\xrightarrow{\mathsf P_g}
 \mathsf P_N\longrightarrow\mathsf S_N\longrightarrow0.
\end{equation}
Since \(f\) and \(g\) are radical and \(f\) is a minimal weak kernel, the sequence \eqref{eq:right-presentation} is the
beginning of a minimal projective resolution of \(\mathsf S_N\). Dually, a
left \(\tau\)-sequence starting at an indecomposable \(L\) gives the
beginning of a minimal projective resolution
\(\mathsf P^N\to\mathsf P^M\to\mathsf P^L\to\mathsf S^L\to0\).

Let \(\cC\) be a \(\tau\)-category. For an indecomposable \(X\), the right
\(\tau\)-sequence ending at \(X\) and the left \(\tau\)-sequence starting
at \(X\) are unique up to isomorphism of complexes
\cite[Section~2.1]{IyamaI}. Following
\cite[Section~1.3]{IyamaIV}, we write them as
\[
 \tau_{\cC}^+X\xrightarrow{\nu_X^+}\theta_{\cC}^+X\xrightarrow{\mu_X^+}X,
 \qquad
 X\xrightarrow{\mu_X^-}\theta_{\cC}^-X\xrightarrow{\nu_X^-}\tau_{\cC}^-X.
\]
We call \(X\) \emph{\(\tau\)-projective} if \(\tau_{\cC}^+X=0\), and
\emph{\(\tau\)-injective} if \(\tau_{\cC}^-X=0\). Following
\cite[Section~2.1(3)]{IyamaI}, we denote the sets of \(\tau\)-projective
and \(\tau\)-injective objects of \(\ind\cC\) by \(\ind_1^+\cC\) and
\(\ind_1^-\cC\). We use the following properties, proved by Iyama.
\begin{proposition}[{\cite[Sections~2.1 and~2.3]{IyamaI}}]\label{rem:tau-properties}
Let \(\cC\) be a \(\tau\)-category, and let \(X\) be an indecomposable.
\begin{enumerate}
\item The map \(\mu_X^+\) is minimal right almost split, and
\(\tau_{\cC}^+X\) is zero or indecomposable. Dually, \(\mu_X^-\) is
minimal left almost split, and \(\tau_{\cC}^-X\) is zero or
indecomposable.
\item If \(\tau_{\cC}^+X\ne0\), then
\(\tau_{\cC}^+X\xrightarrow{\nu_X^+}\theta_{\cC}^+X\xrightarrow{\mu_X^+}X\)
is also the left \(\tau\)-sequence starting at \(\tau_{\cC}^+X\). In
particular, \(\tau_{\cC}^-\tau_{\cC}^+X\cong X\) and
\(\theta_{\cC}^-\tau_{\cC}^+X\cong\theta_{\cC}^+X\). The dual statement
holds if \(\tau_{\cC}^-X\ne0\).
\item The maps \(\tau_{\cC}^+\) and \(\tau_{\cC}^-\) are mutually inverse
bijections between \(\ind\cC\setminus\ind_1^+\cC\) and
\(\ind\cC\setminus\ind_1^-\cC\).
\end{enumerate}
\end{proposition}
\begin{example}\label{ex:module-tau}
The category \(\mod A\) is a \(\tau\)-category: its right and left
\(\tau\)-sequences are the
almost split sequences with the zeros at both ends omitted, together with
\(0\to\rad P\to P\) at projectives and
\(I\to I/\soc I\to0\) at injectives.
The almost split factorization properties give the two exactness
conditions in Definition~\ref{def:tau-category}; the ordinary kernels and cokernels
give the required minimal weak kernels and cokernels. Thus
\(\tau_{\mod A}^+=\tau_A\), \(\tau_{\mod A}^-=\tau_A^-\),
\(\theta_{\mod A}^+=\theta_A\), and \(\theta_{\mod A}^-=\theta_A^-\)
in the notation of Definition~\ref{def:ar}, and the \(\tau\)-projectives
and \(\tau\)-injectives are the projectives and injectives. Each
\(\nu_X^+\) is monic, and each \(\nu_X^-\) is epic. These facts are also noted in
\cite[Section~1.3]{IyamaIV}.
\end{example}

As for modules \cite[VII, Proposition~1.3]{ARS}, the middle terms of
\(\tau\)-sequences record the irreducible maps; for \(\tau\)-categories
this is stated in \cite[Section~9.1]{IyamaI}.
\begin{lemma}\label{lem:tau-multiplicity}
Let \(\cC\) be a Hom-finite \(k\)-linear \(\tau\)-category, and let
\(X\) be an indecomposable. For every indecomposable \(Y\),
\[
 [\theta_{\cC}^+X:Y]_\oplus=\dim_k\Irr_{\cC}(Y,X),\qquad
 [\theta_{\cC}^-X:Y]_\oplus=\dim_k\Irr_{\cC}(X,Y).
\]
\end{lemma}
\begin{proof}
We prove the first formula; the second follows by applying it to the
opposite category.
By \eqref{eq:right-presentation}, \(\mathsf P_{\mu_X^+}\) maps
\(\mathsf P_{\theta_{\cC}^+X}\) onto \(\rad_{\cC}(-,X)\), and its kernel is the image
of \(\mathsf P_{\nu_X^+}\), which lies in \(\rad_{\cC}(-,\theta_{\cC}^+X)\) because
\(\nu_X^+\) is radical. The image of \(\rad_{\cC}(-,\theta_{\cC}^+X)\) is
\(\rad_{\cC}^2(-,X)\): it lies there because \(\mu_X^+\) is radical, and
conversely a composite \(ab\) of radical maps \(b:Y\to W\) and
\(a:W\to X\) has the form \(\mu_X^+(cb)\), where \(a=\mu_X^+c\). Hence
\(\mathsf P_{\mu_X^+}\) induces an isomorphism
\[
 \mathsf P_{\theta_{\cC}^+X}/\rad_{\cC}(-,\theta_{\cC}^+X)\cong
 \rad_{\cC}(-,X)/\rad_{\cC}^2(-,X).
\]
The left side is the direct sum of the functors \(\mathsf S_V\) over the
indecomposable summands \(V\) of \(\theta_{\cC}^+X\), with multiplicity. Evaluate
at \(Y\). Since \(\mathsf S_V(Y)=0\) for \(V\not\cong Y\) and
\(\mathsf S_Y(Y)=\End_{\cC}(Y)/\rad_{\cC}(Y,Y)=k\), the left side has
dimension \([\theta_{\cC}^+X:Y]_\oplus\), and the right side is
\(\Irr_{\cC}(Y,X)\).
\end{proof}

We will repeatedly express radical maps as sums of composites of
irreducible maps. For module categories this is
\cite[V, Theorem~7.8]{ARS}; see also \cite[IV, Corollary~5.6]{ASS}.
\begin{lemma}\label{lem:path-expansion}
Let \(\cC\) be a Hom-finite Krull--Schmidt \(k\)-linear category with finitely
many indecomposable isomorphism classes. Choose maps representing a basis
of each \(\Irr_{\cC}(X,Y)\). Then the radical of \(\cC\) is nilpotent, and every
radical map between indecomposables is a linear combination of nonempty
paths in the chosen maps.
\end{lemma}
\begin{proof}
Put \(M=\bigoplus_{X\in\ind\cC}X\) and \(\Lambda=\End_{\cC}(M)\).
The spaces \(\rad_{\cC}(X,Y)\) are the corresponding matrix blocks of
\(\rad\Lambda\). Since \(\Lambda\) is finite-dimensional, choose \(N\)
with \((\rad\Lambda)^N=0\); then \(\rad_{\cC}^N=0\).
Modulo \(\rad_{\cC}^2\), every radical map
is a linear combination of the chosen representatives. Repeating this
expansion in each factor expresses it modulo \(\rad_{\cC}^{n+1}\)
as a sum of paths of lengths at most \(n\). Taking \(n=N-1\) proves the claim.
\end{proof}

\subsection{Magnitude in a strict \texorpdfstring{\(\tau\)}{tau}-category}\label{subsec:strict-magnitude}
The following definition extends magnitude from module categories to
Krull--Schmidt categories. Proposition~\ref{thm:categorical-magnitude}
below computes it for ideal quotients and categories of poset
representations; we first describe the categories to which it applies.
\begin{definition}\label{def:magnitude}
Let \(\cC\) be a Hom-finite Krull--Schmidt \(k\)-linear category
with finitely many indecomposable isomorphism classes. Its \emph{Hom
matrix} and \emph{magnitude} are
\[
 H_{\cC}=(\dim_k\cC(X,Y))_{X,Y\in\ind\cC},\qquad
 \chi(\cC)=\one^{\mathsf T}H_{\cC}^{-1}\one,
\]
the latter whenever \(H_{\cC}\) is invertible.
\end{definition}

\begin{definition}\label{def:strict-tau}
Following \cite[Section~1.3]{IyamaIV}, a \(\tau\)-category \(\cC\)
is \emph{strict} if \(\nu_X^+\) is monic and \(\nu_X^-\) is epic for every
indecomposable \(X\).
\end{definition}
By Example~\ref{ex:module-tau}, \(\mod A\) is a strict \(\tau\)-category.
In a strict \(\tau\)-category, the presentations
\eqref{eq:right-presentation} become resolutions; cf.
\cite[Section~2.1]{IyamaI} and the proof of
\cite[Theorem~4.3]{IyamaIV}.
\begin{lemma}\label{lem:strict-resolution}
Let \(\cC\) be a strict \(\tau\)-category and \(X\) an indecomposable.
Then the sequences
\[
\begin{gathered}
 0\longrightarrow\mathsf P_{\tau_{\cC}^+X}\longrightarrow
 \mathsf P_{\theta_{\cC}^+X}\longrightarrow\mathsf P_X\longrightarrow
 \mathsf S_X\longrightarrow0,\\
 0\longrightarrow\mathsf P^{\tau_{\cC}^-X}\longrightarrow
 \mathsf P^{\theta_{\cC}^-X}\longrightarrow\mathsf P^X\longrightarrow
 \mathsf S^X\longrightarrow0,
\end{gathered}
\]
induced by the right and left \(\tau\)-sequences at \(X\), are exact.
\end{lemma}
\begin{proof}
By \eqref{eq:right-presentation}, the first sequence is exact except
possibly at \(\mathsf P_{\tau_{\cC}^+X}\), and there it is exact because
\(\nu_X^+\) is monic. The second sequence is dual.
\end{proof}

\begin{proposition}\label{thm:categorical-magnitude}
Let \(\cC\) be a Hom-finite strict \(\tau\)-category with finitely
many indecomposable isomorphism classes.
Let \(a(\cC)=\sum_{X,Y}\dim_k\Irr_{\cC}(X,Y)\).
Then \(H_{\cC}\) is invertible and
\begin{equation}\label{eq:categorical-magnitude}
 \chi(\cC)=2|\ind\cC|-a(\cC)-|\ind_1^+\cC|.
\end{equation}
\end{proposition}
\begin{proof}
The argument is that of \cite[Lemma~1.9]{BHK}, which applies
\cite[Theorem~1.1]{CKL} to these resolutions; the inverse of \(H_{\cC}\)
below also appears in the proof of \cite[Theorem~4.3]{IyamaIV}.
Evaluate the first exact sequence of Lemma~\ref{lem:strict-resolution}
at each indecomposable \(Y\). Let \(M\) be the matrix whose \(X\)-column is
the multiplicity vector of \([X]-[\theta_{\cC}^+X]+[\tau_{\cC}^+X]\), where
\([Z]\) records the multiplicities of the indecomposable summands of \(Z\).
The row of \(H_{\cC}\) indexed by \(Y\) lists the dimensions
\(\dim_k\cC(Y,Z)\), and \(\dim_k\mathsf S_X(Y)\) is \(1\) if \(X=Y\) and
\(0\) otherwise, so the evaluated sequence gives
\[
 (H_{\cC}M)_{Y,X}=\dim_k\cC(Y,X)-\dim_k\cC(Y,\theta_{\cC}^+X)
 +\dim_k\cC(Y,\tau_{\cC}^+X)=\indicator_{\{X=Y\}}.
\]
Thus \(M\) is a right inverse of \(H_{\cC}\), and therefore its inverse.
The total multiplicity of the middle terms is \(a(\cC)\), by
Lemma~\ref{lem:tau-multiplicity}. There are
\(|\ind\cC|-|\ind_1^+\cC|\) nonzero translates, each
indecomposable. Summing all entries gives \eqref{eq:categorical-magnitude}.
\end{proof}
In other words, \(\chi(\cC)\) is the number of vertices of the irreducible
quiver of \(\cC\), minus the number of its arrows, plus the number of
vertices with a nonzero translate.

\subsection{Magnitude of module categories}
Since the \(\tau\)-projectives of \(\mod A\) correspond to the simple
modules, Proposition~\ref{thm:categorical-magnitude} recovers the formula
\eqref{eq:intro-count} of B{\o}rve, Horiatakis, and Kalck. We record it together with the
surplus studied in the rest of the paper. Fix a representation-finite
algebra \(A\).

\begin{definition}\label{def:surplus}
Write \(a(A)=a(\mod A)\), and put \(\sigma(A)=\chi(\mod A)-|A|\). We
call \(\sigma(A)\) the \emph{magnitude surplus}.
\end{definition}

\begin{lemma}[{\cite[Lemma~1.9]{BHK}}]\label{lem:count}
The matrix \(H_A=(\dim_k\Hom_A(X,Y))_{X,Y\in\ind A}\) is invertible, and
\[
 \chi(\mod A)=2|\ind A|-|A|-a(A).
\]
Consequently,
\begin{equation}\label{eq:count}
 \sigma(A)=2(|\ind A|-|A|)-a(A).
\end{equation}
\end{lemma}
\begin{remark}
B{\o}rve, Horiatakis, and Kalck also give a refinement of
\eqref{eq:intro-count}, in which the meshes of \(\Gamma_A\) are counted
according to the number of indecomposable middle terms of the
corresponding almost split sequences \cite[Lemma~1.13]{BHK}.
\end{remark}

The following lemma reduces statements to basic algebras in
Sections~\ref{sec:deletion}--\ref{sec:equality} and computes the surplus
of products of interval algebras in Section~\ref{sec:equality}; cf.
\cite[Remark~1.7(2)]{BHK}.
\begin{lemma}\label{lem:invariance}
Magnitude, \(|A|\), and \(\sigma(A)\) are Morita invariant and additive
on products of algebras.
\end{lemma}
\begin{proof}
A \(k\)-linear Morita equivalence preserves simple classes, indecomposable
classes, and Hom dimensions. For a product of algebras, the Hom matrix is
block diagonal, so its inverse is block diagonal and its entry sum is the
sum for the factors. The simple classes are the disjoint union of those of
the factors. Subtracting their numbers proves the assertion for \(\sigma\).
\end{proof}

\section{Magnitude of Serre subcategories and ideal quotients}\label{sec:quotient}
An ideal quotient removes objects by setting their identity maps equal to
zero. We first describe what this operation does to indecomposables and
irreducible maps in an additive category. For a Serre subcategory of a
module category, we then show that the ideal quotient is a strict
\(\tau\)-category, so that
Proposition~\ref{thm:categorical-magnitude} computes its magnitude, and
compare the almost split sequences of the subcategory with those of the
module category. Theorem~\ref{thm:serre-magnitude} relates \(\mod A\), the subcategory, and
the ideal quotient through two counts \(c\) and \(r\) from this comparison
and a correction term \(\delta\). Theorem~\ref{prop:compensation} then
gives a sufficient condition for the inequality \(r\leq c\).

\subsection{Indecomposables and irreducible maps in ideal quotients}
We recall ideal quotients and describe their indecomposables and
irreducible maps.
\begin{definition}\label{def:ideal-quotient}
Let \(\cC\) be an additive category and \(\mathcal U\) a full
additive subcategory closed under direct summands. The \emph{ideal quotient}
\(\cC/[\mathcal U]\) has the same objects as \(\cC\), and
\[
 (\cC/[\mathcal U])(X,Y)=\cC(X,Y)/[\mathcal U](X,Y),
\]
where \([\mathcal U](X,Y)\) consists of the maps factoring through an
object of \(\mathcal U\). Write \(\pi\) for the quotient functor.
\end{definition}

\begin{proposition}\label{lem:surviving-arrows}
Suppose that \(\cC\) is Hom-finite, \(k\)-linear, and Krull--Schmidt.
Then \(\cC/[\mathcal U]\) is Krull--Schmidt. Its indecomposable
isomorphism classes are indexed by \(\ind\cC\setminus\ind\mathcal U\).
For representatives \(X,Y\) in this complement, \(\pi\) induces
\[
 \Irr_{\cC}(X,Y)\cong\Irr_{\cC/[\mathcal U]}(\pi X,\pi Y).
\]
\end{proposition}
\begin{proof}
For an indecomposable \(X\notin\mathcal U\), the ideal
\([\mathcal U](X,X)\) is proper: otherwise \(X\) would be a summand of
an object of \(\mathcal U\). It is therefore contained in the radical
of the local ring \(\End_{\cC}(X)\). Its quotient is local, so
\(\pi X\) remains indecomposable. If \(\pi X\cong\pi Y\), lift inverse
maps to \(f:X\to Y\) and \(g:Y\to X\). Both \(gf\) and \(fg\) are
identities modulo radical ideals, hence invertible. Thus \(X\cong Y\).
Every object is a finite sum of these objects with local endomorphism
rings, proving the Krull--Schmidt assertion.

For the irreducible maps, Iyama shows for an ideal quotient of an
additive category that
\(\rad_{\cC/[\mathcal U]}^n=(\rad_{\cC}^n+[\mathcal U])/[\mathcal U]\) for
\(n\geq1\) \cite[Section~1.3(5)]{IyamaII}. Let
\(X,Y\in\ind\cC\setminus\ind\mathcal U\). A map \(X\to Y\) which factors
through an indecomposable of \(\mathcal U\) is a composite of two radical
maps, so \([\mathcal U](X,Y)\subseteq\rad_{\cC}^2(X,Y)\). The case
\(n=1,2\) of Iyama's formula therefore gives the asserted isomorphism.
\end{proof}

\subsection{Ideal quotients by Serre subcategories}
We now specialize to a Serre subcategory of a module category.
\begin{setting}\label{set:serre}
Let \(A\) be a finite-dimensional algebra, and \(e\)
an idempotent of \(A\). Put
\[
 B=A/AeA,\qquad \cD=(\mod A)/[\mod B],
\]
where \(\mod B\) is regarded as the full subcategory of \(\mod A\) of
modules annihilated by \(AeA\). Since \(ZAeA=ZeA\) for a module \(Z\),
these are the modules \(Z\) with \(Ze=0\). Since evaluation at \(e\) is
exact, \(\mod B\) is a Serre subcategory, that is, closed under
submodules, quotients, and extensions. Write
\(i:\mod B\hookrightarrow\mod A\) for the inclusion and
\(\pi:\mod A\to\cD\) for the ideal quotient. The
inclusion has the right and left adjoints
\[
 i^!Z=\{z\in Z:zAeA=0\}\cong\Hom_A(B,Z),\qquad
 i^*Z=Z/ZeA\cong Z\otimes_AB,
\]
the largest submodule and the largest quotient of \(Z\) in \(\mod B\).
By extension closure, \(Z/i^!Z\) has no nonzero submodule in \(\mod B\), and
dually \(\ker(Z\to i^*Z)\) has no nonzero quotient in \(\mod B\).
\end{setting}
Every Serre subcategory of \(\mod A\) arises in this way: choose \(e\) as
the sum of primitive idempotents corresponding to the simple modules
outside it. For a module map \(f:X\to Y\), write \(f_e:Xe\to Ye\) for its
restriction. Ringel and Vossieck show that a map of \(A\)-modules factors
through \(\mod B\) exactly when its restriction \(f_e\) is zero
\cite[Section~10, Lemma~1]{RV}; this is the faithfulness in the following
lemma.

\begin{lemma}\label{lem:quotient-evaluation}
Evaluation at \(e\) induces a faithful functor
\(F:\cD\to\mod eAe\), \(\pi X\mapsto Xe\). In particular, a map in \(\cD\)
is monic, respectively epic, if its image under \(F\) is.
\end{lemma}
\begin{proof}
A map \(f:X\to Y\) with \(f_e=0\) kills \(XeA\), so it factors through
\(X/XeA\in\mod B\). Conversely, a map factoring through \(\mod B\) has zero
evaluation at \(e\). Faithful functors reflect monomorphisms and
epimorphisms.
\end{proof}

\begin{remark}\label{ex:a2-quotients}
The ideal quotient \(\cD\) differs from the abelian Serre quotient
\((\mod A)/(\mod B)\), in which every map with kernel and cokernel in \(\mod B\)
becomes invertible. Let \(A=(k(1\to2))^{\op}\) and \(e=e_2\), so that
\(\mod B=\add S_1\). The inclusion \(S_2\to P_1\) has cokernel \(S_1\), so it
becomes invertible in \((\mod A)/(\mod B)\). Under \(F\) it becomes the
identity of \(k\), so it is monic and epic in \(\cD\). It is not invertible
in \(\cD\), because \(\Hom_A(P_1,S_2)=0\).
\end{remark}

To find the \(\tau\)-sequences of \(\cD\), we use Iyama's description of
\(\tau\)-sequences in an ideal quotient.
\begin{theorem}[{\cite[Section~1.4(2)]{IyamaII}}]\label{thm:iyama-quotient}
Let \(\cC\) be a \(\tau\)-category and \(\mathcal U\) a full
additive subcategory closed under direct summands, with quotient functor
\(\pi:\cC\to\cC/[\mathcal U]\). Then
\(\cC/[\mathcal U]\) is a \(\tau\)-category. Let
\(X\in\ind\cC\setminus\ind\mathcal U\). If
\(\pi(\theta_{\cC}^+X)\ne0\), the right \(\tau\)-sequence ending at
\(\pi X\) is
\[
 \pi(\tau_{\cC}^+X)\xrightarrow{\pi\nu_X^+}\pi(\theta_{\cC}^+X)
 \xrightarrow{\pi\mu_X^+}\pi X;
\]
otherwise it is \(0\to0\to\pi X\).
\end{theorem}
Applied to the opposite categories, the theorem gives the dual description
of left \(\tau\)-sequences.

\begin{proposition}\label{prop:quotient-sequences}
The category \(\cD\) in Setting~\ref{set:serre} is a strict
\(\tau\)-category. For \(X\in\ind A\setminus\ind B\), the right
\(\tau\)-sequence ending at \(\pi X\) and the left \(\tau\)-sequence
starting at \(\pi X\) are
\[
 \pi(\tau_A X)\longrightarrow\pi(\theta_A X)\longrightarrow\pi X,
 \qquad
 \pi X\longrightarrow\pi(\theta_A^-X)\longrightarrow\pi(\tau_A^-X).
\]
In particular, \(\tau_{\cD}^+(\pi X)\cong\pi(\tau_A X)\) and
\(\tau_{\cD}^-(\pi X)\cong\pi(\tau_A^-X)\).
\end{proposition}
\begin{proof}
Apply Theorem~\ref{thm:iyama-quotient} to the \(\tau\)-category
\(\mod A\) of Example~\ref{ex:module-tau} and \(\mathcal U=\mod B\).
If \(X\) is nonprojective, exactness of \((-)e\) gives
\((\theta_A X)e\twoheadrightarrow Xe\ne0\); hence
\(\pi(\theta_A X)\ne0\), and the right \(\tau\)-sequence of \(\pi X\) is
the image of the almost split sequence ending at \(X\). If \(X\) is
projective, then \(\tau_A X=0\) and \(\theta_A X=\rad X\), and both
cases of the theorem give the sequence \(0\to\pi(\rad X)\to\pi X\). This proves the
description of right \(\tau\)-sequences; that of left \(\tau\)-sequences
is dual.
The first maps \(\pi\nu_X^+\) of these right \(\tau\)-sequences become
monic under \(F\), by exactness of \((-)e\), and the last maps
\(\pi\nu_X^-\) of the left \(\tau\)-sequences become epic.
Lemma~\ref{lem:quotient-evaluation} proves strictness.
\end{proof}
When \(A\) is representation-finite, the strictness of \(\cD\) is also
a case of \cite[Theorem~6.1]{IyamaII}, since \(i^!\) and \(i^*\) have
monic counit and epic unit.

\subsection{Restricting almost split sequences to Serre subcategories}\label{subsec:restriction}
By Proposition~\ref{prop:quotient-sequences}, the translate
\(\tau_{\cD}^+(\pi X)\) is \(\pi(\tau_A X)\), which is zero exactly when
\(\tau_A X\) lies in \(\mod B\). The translate of a
\(B\)-module in \(\mod B\) is obtained differently: by the following
result of Hoshino, it is the largest submodule of its translate in
\(\mod A\) that belongs to \(\mod B\).

\begin{proposition}[{\cite[Section~1, Lemmas~2(2) and~3(2)]{Hos}}]\label{lem:relative-ar}
Let \(Y\) be a nonprojective indecomposable of \(\mod B\), and write its
almost split sequence in \(\mod A\) as
\(0\to\tau_A Y\to\theta_A Y\to Y\to0\). Then
\begin{equation}\label{eq:relative-ar}
 0\longrightarrow i^!(\tau_A Y)\longrightarrow i^!(\theta_A Y)
 \longrightarrow Y\longrightarrow0
\end{equation}
is its almost split sequence in \(\mod B\). In particular,
\(\tau_BY\cong i^!(\tau_A Y)\). Dually, applying \(i^*\) to the
ambient almost split sequence starting at a noninjective \(X\in\ind B\)
gives its sequence in \(\mod B\), with
\(\tau_B^-X\cong i^*(\tau_A^-X)\).
\end{proposition}
\begin{proof}
The lemmas of Hoshino concern a torsion pair in a module category, and
they are essentially due to Auslander and Smal{\o} \cite{AS}. Lemma~2(2) of
\cite{Hos} states the following. Let \(\mathcal T\) be a torsion class with
torsion radical \(t\), and let \(Y\in\mathcal T\) be indecomposable and not
Ext-projective in \(\mathcal T\), that is, \(\Ext_A^1(Y,\mathcal T)\ne0\). Then applying \(t\) to the almost split
sequence of \(\mod A\) ending at \(Y\) gives an almost split sequence in
\(\mathcal T\). Lemma~3(2) is the dual statement for a torsion-free class and
its torsion-free quotient functor.

The subcategory \(\mod B\) is closed under quotients and extensions, so it
is a torsion class, with torsion radical \(i^!\). It is closed under
submodules and extensions, so it is also the torsion-free class of another
torsion pair, with torsion-free quotient functor \(i^*\). Since \(\mod B\)
is extension-closed, \(\Ext_B^1=\Ext_A^1\) on \(\mod B\), so its Ext-projective and Ext-injective objects are the
projective and injective \(B\)-modules. The two lemmas therefore give the
two assertions.
\end{proof}

\begin{example}\label{ex:fork-restriction}
Let \(A=(k(1\to2\leftarrow3))^{\op}\) and \(e=e_3\), so that \(\mod B\)
consists of the representations vanishing at \(3\). The ambient sequence ending at \(S_1\) is
\[
 0\longrightarrow P_3\longrightarrow I_2\longrightarrow S_1\longrightarrow0.
\]
Applying \(i^!\) gives
\[
 0\longrightarrow S_2\longrightarrow P_1\longrightarrow S_1\longrightarrow0,
\]
the almost split sequence in \(\mod B\simeq\mod\bigl((k(1\to2))^{\op}\bigr)\).
In particular \(\tau_A S_1=P_3\), whereas \(\tau_BS_1=S_2\). The map
\(P_1\to S_1\) is irreducible in \(\mod B\), but factors through \(I_2\)
in \(\mod A\).
\end{example}

\subsection{Magnitudes of a Serre subcategory and its ideal quotient}\label{subsec:new-counts}
We now compare the irreducible maps and the almost split sequences of
\(\mod B\) with those of \(\mod A\), and use the comparison to relate
the magnitudes of \(\mod A\), \(\mod B\), and \(\cD\).
Assume from now on in this section that \(B\) is representation-finite.

\begin{definition}\label{def:new-counts}
For \(V,W\in\ind B\), put
\[
 c_{VW}=\dim_k\Irr_B(V,W)-\dim_k\Irr_A(V,W),\qquad
 c=\sum_{V,W\in\ind B}c_{VW}.
\]
Let \(r\) be the number of isomorphism classes of almost split sequences
in \(\mod B\) that are not almost split in \(\mod A\).
\end{definition}

\begin{lemma}\label{lem:new-arrows}
The identity on \(\rad_B(V,W)=\rad_A(V,W)\) induces a surjection
\(\Irr_B(V,W)\twoheadrightarrow\Irr_A(V,W)\).
Its kernel is \(\rad_A^2(V,W)/\rad_B^2(V,W)\) and has dimension
\(c_{VW}\geq0\).
\end{lemma}
\begin{proof}
Fullness gives equality of the radicals. Every factorization through
\(\mod B\) is also one through \(\mod A\), so
\(\rad_B^2(V,W)\subseteq\rad_A^2(V,W)\). Taking the two quotients
proves the assertion.
\end{proof}
Thus \(\Gamma_B\) has \(c_{VW}\) more arrows from \(V\) to \(W\) than
\(\Gamma_A\) has.

The following identity expresses
\(\chi(\mod A)-\chi(\mod B)-\chi(\cD)\) through \(c\), \(r\), and the middle terms of the almost split sequences of
\(\mod A\) with exactly one end term in \(\mod B\).

\begin{theorem}\label{thm:serre-magnitude}
Assume \(A\) is representation-finite. For each almost split sequence
of \(\mod A\) with exactly one end term in \(\mod B\), let \(m_{\mathrm{out}}\) be the
number of its indecomposable middle summands outside \(\mod B\), counting
multiplicities, and put
\[
 \delta=\sum_{\text{such sequences}}(m_{\mathrm{out}}-1).
\]
Then \(\delta\geq0\), and
\begin{equation}\label{eq:serre-magnitude}
 \chi(\mod A)-\chi(\mod B)-\chi(\cD)=c-r-\delta.
\end{equation}
\end{theorem}
\begin{proof}
Let \(0\to\tau_A Y\to\theta_A Y\to Y\to0\) be a sequence in the sum.
Since \(Ye\ne0\) or \((\tau_A Y)e\ne0\), applying the exact functor
\((-)e\) gives \((\theta_A Y)e\ne0\). Thus \(\theta_A Y\) has an
indecomposable summand outside \(\mod B\), and
\(m_{\mathrm{out}}\geq1\).
By Proposition~\ref{lem:surviving-arrows},
\(|\ind A|=|\ind\cD|+|\ind B|\). Let \(t\) be the number of sequences in
the sum.

We count the \(|\ind A|-|A|\) almost split sequences of \(\mod A\) by their
end terms. A sequence with both end terms in \(\mod B\) lies in \(\mod B\), by
extension closure, and remains almost split there. These are the
almost split sequences of \(\mod B\) not counted by \(r\), and there are
\(|\ind B|-|B|-r\) of them. By
Proposition~\ref{prop:quotient-sequences}, the sequences with both
end terms outside \(\mod B\) correspond to the objects of \(\cD\) with
nonzero translate, and there are \(|\ind\cD|-|\ind_1^+\cD|\) of them. Hence
\[
 |\ind A|-|A|=(|\ind\cD|-|\ind_1^+\cD|)+(|\ind B|-|B|-r)+t,
\]
and substituting \(|\ind A|=|\ind\cD|+|\ind B|\) gives
\begin{equation}\label{eq:count-t}
 t=|\ind_1^+\cD|-(|A|-|B|)+r.
\end{equation}

Let \(z\) be the number of arrows of \(\Gamma_A\) with exactly one
endpoint in \(\mod B\), counted with multiplicity. Let \(V\to W\) be such
an arrow with \(V\in\mod B\). An arrow out of an injective \(A\)-module in \(\mod B\) is epic, so
its target also lies in \(\mod B\); hence \(V\) is not injective. The almost
split sequence starting at \(V\) has \(W\) as a middle summand, so its last
term is not in \(\mod B\), by extension closure. Thus this sequence is counted
by \(t\), and \(W\) is one of the middle summands outside \(\mod B\) counted
by \(m_{\mathrm{out}}\) for this sequence. Dually, such an arrow \(W\to V\) with \(V\in\mod B\) ends at a
nonprojective \(V\) and gives a middle summand outside \(\mod B\) of the
sequence ending at \(V\). Conversely, every middle summand outside \(\mod B\)
of a sequence counted by \(t\) gives an arrow to or from its end term in
\(\mod B\). Since such a sequence has only one end term in \(\mod B\), each
such arrow arises from exactly one sequence, the one starting or ending at
its endpoint in \(\mod B\). We therefore obtain a bijection between the
arrows of \(\Gamma_A\) with exactly one endpoint in \(\mod B\) and the
pairs of a sequence counted by \(t\) and one of its middle summands
outside \(\mod B\). By Proposition~\ref{prop:ar-multiplicity} it
respects multiplicities, so \(z=\sum m_{\mathrm{out}}\), and hence
\begin{equation}\label{eq:delta-z}
 \delta=z-t.
\end{equation}
The arrows of \(\Gamma_A\) between objects outside \(\mod B\) are the
arrows of \(\cD\) by Proposition~\ref{lem:surviving-arrows}; those between
objects of \(\mod B\) number \(a(B)-c\) by Definition~\ref{def:new-counts};
the remaining \(z\) arrows have exactly one endpoint in \(\mod B\). Hence
\begin{equation}\label{eq:arrows-split}
 a(A)=a(\cD)+z+a(B)-c.
\end{equation}
Lemma~\ref{lem:count} and Proposition~\ref{thm:categorical-magnitude} give
\begin{equation}\label{eq:three-magnitudes}
\begin{gathered}
 \chi(\mod A)=2|\ind A|-|A|-a(A),\qquad
 \chi(\mod B)=2|\ind B|-|B|-a(B),\\
 \chi(\cD)=2|\ind\cD|-a(\cD)-|\ind_1^+\cD|.
\end{gathered}
\end{equation}
Subtract the last two formulas in \eqref{eq:three-magnitudes} from the
first, and substitute \(|\ind A|=|\ind\cD|+|\ind B|\) and
\eqref{eq:arrows-split}. Then \eqref{eq:count-t} and \eqref{eq:delta-z}
give
\[
 \chi(\mod A)-\chi(\mod B)-\chi(\cD)=|\ind_1^+\cD|-(|A|-|B|)+c-z
 =(t-r)+c-z=c-r-\delta.
\]
\end{proof}
For a primitive idempotent of a representation-directed algebra, we show
in Proposition~\ref{thm:directed-deletion} that \(\delta=0\), using
Proposition~\ref{prop:boundary-one}. We continue
Example~\ref{ex:fork-restriction} and compute every term of
Theorem~\ref{thm:serre-magnitude}.
\begin{example}\label{ex:fork-identity}
Keep \(A=(k(1\to2\leftarrow3))^{\op}\) and \(e=e_3\). The
Auslander--Reiten quiver of \(A\) is
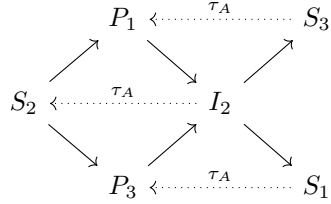
\[
\begin{tikzcd}[column sep=1.8em,row sep=1.6em]
 &P_1\arrow[dr]&&S_3\arrow[ll,dotted,"\tau_A"']\\
 S_2\arrow[ur]\arrow[dr]&&I_2\arrow[ur]\arrow[dr]
   \arrow[ll,dotted,"\tau_A"']&\\
 &P_3\arrow[ur]&&S_1\arrow[ll,dotted,"\tau_A"']
\end{tikzcd}
\]
where \(S_2=P_2\), \(S_1=I_1\), and \(S_3=I_3\). It has six vertices and
six arrows, so \(\chi(\mod A)=12-3-6=3\) by Lemma~\ref{lem:count}. The
category \(\mod B\) has the three indecomposables \(S_2\), \(P_1\),
\(S_1\) and the two arrows \(S_2\to P_1\to S_1\), so
\(\chi(\mod B)=6-2-2=2\).

The indecomposables of \(\cD\) are \(P_3\), \(I_2\), and \(S_3\), with the
two arrows \(P_3\to I_2\to S_3\). By
Proposition~\ref{prop:quotient-sequences}, \(\tau_{\cD}^+I_2\) and
\(\tau_{\cD}^+S_3\) are the images of \(S_2\) and \(P_1\), which are zero
in \(\cD\), and \(P_3\) is projective. Thus all three objects are
\(\tau\)-projective, although \(I_2\) and \(S_3\) are not projective over
\(A\), and Proposition~\ref{thm:categorical-magnitude} gives
\(\chi(\cD)=6-2-3=1\).

Of the two arrows of \(\Gamma_B\), the arrow \(S_2\to P_1\) is an arrow of
\(\Gamma_A\), and \(P_1\to S_1\) is not; hence \(c=1\). The only almost
split sequence of \(\mod B\) is not almost split in \(\mod A\), so
\(r=1\). The almost split sequences of \(\mod A\) with exactly one
end term in \(\mod B\) end at \(I_2\), \(S_1\), and \(S_3\). Each has exactly
one middle summand outside \(\mod B\), namely \(P_3\), \(I_2\), and
\(I_2\), so \(\delta=0\). Both sides of \eqref{eq:serre-magnitude} are
zero.
\end{example}

\subsection{Almost split sequences not preserved by the inclusion}\label{subsec:extension-supports}
In the identity \eqref{eq:serre-magnitude}, the difference \(c-r\) is
nonnegative under a condition on the extensions between \(\mod B\) and
the simple modules outside it.
Let \(E\) be a direct sum of one representative of each simple module
outside \(\mod B\). Define functors on \(\mod B\) by
\[
 \mathsf H_E(W)=\Ext_A^1(E,W),\qquad \mathsf G_E(W)=\Ext_A^1(W,E).
\]

\begin{lemma}\label{lem:extension-supports}
Let \(P_E\twoheadrightarrow E\) and \(E\hookrightarrow I_E\) be the
projective cover and injective envelope, and let
\(\xi:0\to X\to\theta_B Y\to Y\to0\) be an almost split sequence in
\(\mod B\).
Put \(V=\tau_A Y\) and \(Q=V/i^!V\).
\begin{enumerate}
\item \(\mathsf H_E\cong\Hom_B(i^*\rad P_E,-)\) and
\(\mathsf G_E\cong\Hom_B(-,i^!(I_E/E))\).
\item \(\Ext_A^1(E,V)=0\), and there is an exact sequence
\begin{equation}\label{eq:extension-support-sequence}
 0\longrightarrow\Dk \mathsf G_E(Y)\longrightarrow\Hom_A(E,Q)
 \longrightarrow \mathsf H_E(X)\longrightarrow0.
\end{equation}
\item The sequence \(\xi\) is not almost split in \(\mod A\) if and only
if \(\mathsf H_E(X)\ne0\) or \(\mathsf G_E(Y)\ne0\).
\end{enumerate}
\end{lemma}
\begin{proof}
(1) For \(W\in\mod B\), the top of \(P_E\) and socle of \(I_E\) give
\(\Hom_A(P_E,W)=0=\Hom_A(W,I_E)\). Applying \(\Hom_A(-,W)\) to
\(0\to\rad P_E\to P_E\to E\to0\) gives
\(\Ext_A^1(E,W)\cong\Hom_A(\rad P_E,W)\cong\Hom_B(i^*\rad P_E,W)\), by the
adjunction for \(i^*\). Dually, applying \(\Hom_A(W,-)\) to
\(0\to E\to I_E\to I_E/E\to0\) and using the adjunction for \(i^!\) gives
the second representation.

(2) By Proposition~\ref{lem:relative-ar}, identify \(X=i^!V\).
The first isomorphism in \eqref{eq:ar-duality}, with \(M=E\) and \(N=V\),
gives
\[
 \Ext_A^1(E,V)=\Dk\uhom_A(Y,E)=0,
\]
since \(\Hom_A(Y,E)=0\) for \(Y\in\mod B\).
The second, with \(M=Y\) and \(N=E\), gives
\(\Dk \mathsf G_E(Y)=\ohom_A(E,V)=\Hom_A(E,V)\).
For the equality \(\ohom_A(E,V)=\Hom_A(E,V)\), let a nonzero map
\(S\to V\) from a simple summand \(S\) of \(E\) factor through an
injective module. It then factors through the injective envelope
\(S\hookrightarrow I(S)\). The resulting map \(I(S)\to V\) is monic on the
essential simple socle \(S\),
hence monic and split. This contradicts the fact that the
indecomposable \(V\) is not injective. Applying \(\Hom_A(E,-)\) to
\(0\to X\to V\to Q\to0\), and using \(\Hom_A(E,X)=0\), proves
\eqref{eq:extension-support-sequence}.

(3) The sequence \(\xi\) is almost split in \(\mod A\) precisely when
\(V\in\mod B\), equivalently when \(Q=0\). Indeed, if \(\xi\) is almost
split in \(\mod A\), then \(V\cong X\in\mod B\). Conversely, if
\(V\in\mod B\), then the ambient sequence lies in \(\mod B\) by extension
closure, and it coincides with \(\xi\) by
Proposition~\ref{lem:relative-ar}. Every simple subobject of a nonzero \(Q\)
is outside \(\mod B\), by Setting~\ref{set:serre}.
Thus \(Q\ne0\) is equivalent to \(\Hom_A(E,Q)\ne0\), and (2) proves the
criterion.
\end{proof}

\begin{theorem}\label{prop:compensation}
Suppose that
\begin{equation}\label{eq:disjoint-supports}
 \mathsf H_E(W)=0\quad\text{or}\quad \mathsf G_E(W)=0
 \qquad\text{for every }W\in\ind B.
\end{equation}
Then \(r\leq c\). More precisely, the almost split sequences of
\(\mod B\) that are not almost split in \(\mod A\) can be assigned distinct
ordered pairs \((V,W)\) with \(c_{VW}>0\).
\end{theorem}
\begin{proof}
An irreducible map \(V\to W\) in the module category \(\mod B\) is
monic or epic. In the monic case, representability of \(\mathsf H_E\)
(Lemma~\ref{lem:extension-supports}(1)) gives an
injection \(\mathsf H_E(V)\hookrightarrow \mathsf H_E(W)\); in the epic case,
representability of \(\mathsf G_E\) gives an injection \(\mathsf G_E(W)\hookrightarrow \mathsf G_E(V)\).
Hence, if \(\Irr_B(V,W)\ne0\), \(\mathsf H_E(V)\ne0\), and
\(\mathsf G_E(W)\ne0\), then \(\mathsf H_E(W)\ne0\ne\mathsf G_E(W)\) in the
monic case and \(\mathsf H_E(V)\ne0\ne\mathsf G_E(V)\) in the epic case.
Thus \eqref{eq:disjoint-supports} implies
\begin{equation}\label{eq:forbidden-arrow}
 \Irr_B(V,W)\ne0\quad\Longrightarrow\quad
 \mathsf H_E(V)=0\ \text{or}\ \mathsf G_E(W)=0.
\end{equation}

Consider an almost split sequence \(0\to X\to\theta_B Y\to Y\to0\) of
\(\mod B\) that is not almost split in \(\mod A\).
Since \(\Hom_A(E,Y)=0\), its long exact sequence gives
\(\mathsf H_E(X)\hookrightarrow \mathsf H_E(\theta_B Y)\). If \(\mathsf H_E(X)\ne0\), some middle summand
\(W\) has \(\mathsf H_E(W)\ne0\), and its arrow to \(Y\), together with
\eqref{eq:forbidden-arrow}, forces \(\mathsf G_E(Y)=0\).
By Lemma~\ref{lem:extension-supports}(3), at least one of \(\mathsf H_E(X)\)
and \(\mathsf G_E(Y)\) is nonzero, so exactly one is. We assign to the
sequence an ordered pair of indecomposable \(B\)-modules, in two cases.

\textbf{Case 1: \(\mathsf H_E(X)\ne0\).}
Write the ambient middle term ending at \(Y\) as
\(\theta_A Y=M_{\mathrm{in}}\oplus M_{\mathrm{out}}\), where the summands of
\(M_{\mathrm{in}}\) lie in \(\mod B\) and none of those of
\(M_{\mathrm{out}}\) do. Proposition~\ref{lem:relative-ar} gives
\begin{equation}\label{eq:restricted-middle}
 \theta_B Y=M_{\mathrm{in}}\oplus i^!M_{\mathrm{out}}.
\end{equation}
Let \(W\in\ind B\). By Proposition~\ref{prop:ar-multiplicity} in
\(\mod B\) and in \(\mod A\),
\([\theta_B Y:W]_\oplus=\dim_k\Irr_B(W,Y)\) and
\([M_{\mathrm{in}}:W]_\oplus=[\theta_A Y:W]_\oplus=\dim_k\Irr_A(W,Y)\).
Subtracting and using \eqref{eq:restricted-middle} gives
\begin{equation}\label{eq:restricted-multiplicity}
 [i^!M_{\mathrm{out}}:W]_\oplus=c_{WY}.
\end{equation}
The component \(X\to M_{\mathrm{in}}\) of the first map of \(\xi\) is the
restriction of \(\nu_Y^+:\tau_A Y\to\theta_A Y\) to \(X=i^!(\tau_A Y)\), followed by the
projection to \(M_{\mathrm{in}}\). Hence the map
\(\mathsf H_E(X)\to\mathsf H_E(M_{\mathrm{in}})\) factors through
\(\Ext_A^1(E,\tau_A Y)=0\), by Lemma~\ref{lem:extension-supports}(2).
The injection into \(\mathsf H_E(\theta_B Y)\) therefore provides a
summand \(W\) of \(i^!M_{\mathrm{out}}\) with \(\mathsf H_E(W)\ne0\).
Assign to the sequence one such pair \((W,Y)\), for which \(c_{WY}>0\).

\textbf{Case 2: \(\mathsf G_E(Y)\ne0\).}
The dual argument applies to the sequence starting at \(X\). Write \(\theta_A^-X=M'_{\mathrm{in}}\oplus M'_{\mathrm{out}}\) as
above. Since \(\xi\) is also the almost split sequence of \(\mod B\) starting at
\(X\), the dual part of Proposition~\ref{lem:relative-ar} gives
\(\theta_B Y=M'_{\mathrm{in}}\oplus i^*M'_{\mathrm{out}}\) and
\([i^*M'_{\mathrm{out}}:W]_\oplus=c_{XW}\). The map
\(\mathsf G_E(Y)\to\mathsf G_E(M'_{\mathrm{in}})\) factors through
\(\Ext_A^1(\tau_A^-X,E)\), which vanishes by the second isomorphism in
\eqref{eq:ar-duality} because \(\Hom_A(E,X)=0\). This assigns a pair
\((X,W)\) with \(c_{XW}>0\) and \(\mathsf G_E(W)\ne0\).

The pairs of Case~1 have distinct targets \(Y\), and those of Case~2
distinct sources \(X\), by uniqueness of almost split sequences. A pair
\((V',W')\) chosen in both cases would have
\(\mathsf H_E(V')\ne0\) and \(\mathsf G_E(W')\ne0\), contrary to
\eqref{eq:forbidden-arrow}.
Thus all assigned pairs are distinct, and their number is at most
\(\sum_{V,W}c_{VW}=c\).
\end{proof}

\begin{corollary}\label{cor:directed-compensation}
If \(A\) is representation-finite and representation-directed and
\(e\) is a primitive idempotent, then \(r\leq c\).
\end{corollary}
\begin{proof}
Here \(E\) is a single simple module. If \(\Ext_A^1(M,N)\ne0\) for
indecomposables, the second isomorphism in \eqref{eq:ar-duality} gives a
nonzero map \(N\to\tau_A M\). Following it by the two maps through a middle
summand \(Z\) of the almost split sequence ending at \(M\) gives nonzero
maps \(N\to\tau_A M\to Z\to M\) between indecomposables. Removing an isomorphism step if present leaves nonzero
nonisomorphisms. For \(W\in\ind B\), nonzero \(\Ext_A^1(E,W)\) and
\(\Ext_A^1(W,E)\) would therefore give a cycle. Condition~\eqref{eq:disjoint-supports}
holds, and Theorem~\ref{prop:compensation} applies.
\end{proof}

In Example~\ref{ex:fork-identity}, the assignment of
Theorem~\ref{prop:compensation} can be computed directly.
\begin{example}\label{ex:fork-compensation}
Keep \(A=(k(1\to2\leftarrow3))^{\op}\) and \(e=e_3\), so that
\(E=S_3\). Since \(S_3\) is injective, \(\mathsf G_E=0\), and
\eqref{eq:disjoint-supports} holds. Since \(\rad P_3=S_2\), we have
\(\mathsf H_E(S_2)=\Ext_A^1(S_3,S_2)\ne0\). Thus the almost split
sequence \(0\to S_2\to P_1\to S_1\to0\) of \(\mod B\) has
\(\mathsf H_E(X)\ne0\) for \(X=S_2\). In the proof of
Theorem~\ref{prop:compensation}, \(\theta_A S_1=I_2\) has no summand in
\(\mod B\), and \(i^!I_2=P_1\). The proof therefore assigns to this
sequence the pair \((P_1,S_1)\), with \(c_{P_1S_1}=1\).
\end{example}

\section{Poset representations and deletion of a simple module}\label{sec:deletion}
We compute the magnitude of a representation-finite category of poset
representations. The equivalence of Ringel and Vossieck
(Proposition~\ref{prop:poset-equivalence}) then applies this result to the
ideal quotient associated with a primitive idempotent in a directed
algebra. Combined with Theorem~\ref{thm:serre-magnitude} and
Corollary~\ref{cor:directed-compensation}, it proves
Proposition~\ref{thm:directed-deletion}, which compares \(\sigma(A)\) with
\(\sigma(A/AeA)\).

Three results on the quotient use the work of Ringel and Vossieck.
Proposition~\ref{prop:boundary-one}, on \(\dim_k Xe\) at the
\(\tau\)-projectives and \(\tau\)-injectives \(X\) of \(\cD\), follows from
their bound on \(\dim_k Xe\) for indecomposable \(A\)-modules. They proved the equivalence
\(\cD\simeq\TSp\) (Proposition~\ref{prop:poset-equivalence}). The
inequality \(\chi(\cD)\geq1\) (Corollary~\ref{prop:intrinsic}) follows
from this equivalence and Theorem~\ref{thm:poset-magnitude}, whose proof
uses their results on \(T\)-spaces.
Appendix~\ref{app:structural} proves all three directly, so
Proposition~\ref{thm:directed-deletion} does not depend on \cite{RV}.

\subsection{Magnitude of poset representations}\label{subsec:poset}
We first define representations of a poset by subspaces and compute the
magnitude of the category they form.
\begin{definition}\label{def:t-space}
For a finite poset \(T\), a \emph{representation of \(T\) by subspaces},
or \(T\)-space, is a finite-dimensional vector space \(W\), its
\emph{total space}, with subspaces \(W_t\), its \emph{indexed subspaces},
such that \(W_s\subseteq W_t\) whenever \(s\leq t\).
A morphism \((W;W_t)\to(W';W'_t)\) is a linear map \(f:W\to W'\)
such that \(f(W_t)\subseteq W'_t\) for every \(t\in T\). Write
\(\TSp\) for this category. A subset \(U\subseteq T\) is
\emph{upward-closed} if \(s\in U\) and \(s\leq t\) imply \(t\in U\).
For such \(U\), let \(K_U\) have total space \(k\), with subspace
\((K_U)_t=k\) for \(t\in U\) and \((K_U)_t=0\) otherwise.
\end{definition}
In the proof of Theorem~\ref{thm:poset-magnitude} below, we construct a
function \(h\) on the indecomposable \(T\)-spaces that increases by one
along arrows and satisfies \(h(X)=h(\tau^+X)+2\) for nonzero translates.
The following lemma computes the number of arrows, and hence the
magnitude, from such a function.
\begin{lemma}\label{lem:height-arrows}
Let \(\cC\) be a Hom-finite strict \(\tau\)-category with finitely
many indecomposable isomorphism classes. Suppose there are
indecomposables \(X_0,X_1\) and a function
\(h:\ind\cC\to\mathbb Z\) with the following properties.
\begin{enumerate}
\item The object \(X_0\) is \(\tau\)-projective and has no incoming
arrow, and every other \(\tau\)-projective has exactly one incoming arrow,
counted with multiplicity. Dually, \(X_1\) is \(\tau\)-injective and has
no outgoing arrow, and every other \(\tau\)-injective has exactly one
outgoing arrow.
\item One has \(h(X_0)=0\), \(h(Y)=h(X)+1\) for every arrow
\(X\to Y\), and \(h(X)=h(\tau_{\cC}^+X)+2\) whenever
\(\tau_{\cC}^+X\ne0\).
\end{enumerate}
Then, with \(L=h(X_1)\),
\[
 a(\cC)=2|\ind\cC|-L-2,\qquad \chi(\cC)=1+L-(|\ind_1^+\cC|-1).
\]
\end{lemma}
\begin{proof}
By Proposition~\ref{rem:tau-properties}(3), \(\tau_{\cC}^+\) and
\(\tau_{\cC}^-\) are inverse bijections between
\(\ind\cC\setminus\ind_1^+\cC\) and
\(\ind\cC\setminus\ind_1^-\cC\), so
\(|\ind_1^-\cC|=|\ind_1^+\cC|\). For \(X\notin\ind_1^+\cC\), the right
\(\tau\)-sequence at \(X\) is the left \(\tau\)-sequence at
\(\tau_{\cC}^+X\), by part~(2) of the same proposition. Lemma~\ref{lem:tau-multiplicity}
therefore gives
\begin{equation}\label{eq:paired-degrees}
 \indeg_{\cC}(X)=\outdeg_{\cC}(\tau_{\cC}^+X).
\end{equation}
Let \(h_+\) and \(h_-\) be the sums of \(h\) over
\(\ind_1^+\cC\) and \(\ind_1^-\cC\). Summing
\(h(X)-h(\tau_{\cC}^+X)=2\) over \(X\notin\ind_1^+\cC\)
gives \(h_--h_+=2(|\ind\cC|-|\ind_1^+\cC|)\).

Each arrow raises \(h\) by one, so
\(a(\cC)=\sum_X h(X)\bigl(\indeg_{\cC}(X)-\outdeg_{\cC}(X)\bigr)\).
In the incoming sum, the \(\tau\)-projectives contribute \(h_+\), by~(1)
and \(h(X_0)=0\). For the other \(X\), write
\(h(X)=h(\tau_{\cC}^+X)+2\), use \eqref{eq:paired-degrees}, and
reindex by \(Y=\tau_{\cC}^+X\), which runs over
\(\ind\cC\setminus\ind_1^-\cC\). In the outgoing sum, the
\(\tau\)-injectives contribute \(h_--L\), by~(1). Also by~(1), the
\(\tau\)-injectives other than \(X_1\) have one outgoing arrow each, and
\(|\ind_1^-\cC|=|\ind_1^+\cC|\), so
\(\sum_{Y\notin\ind_1^-\cC}\outdeg_{\cC}(Y)=a(\cC)-(|\ind_1^+\cC|-1)\). Hence
\[
\begin{aligned}
 a(\cC)
 &=h_++\sum_{Y\notin\ind_1^-\cC}(h(Y)+2)\outdeg_{\cC}(Y)
   -\sum_{Y\notin\ind_1^-\cC}h(Y)\outdeg_{\cC}(Y)-(h_--L)\\
 &=2\bigl(a(\cC)-(|\ind_1^+\cC|-1)\bigr)+h_+-h_-+L\\
 &=2\bigl(a(\cC)-|\ind_1^+\cC|+1\bigr)-2(|\ind\cC|-|\ind_1^+\cC|)+L,
\end{aligned}
\]
and solving gives
\(a(\cC)=2|\ind\cC|-L-2\). Proposition~\ref{thm:categorical-magnitude} gives
\(\chi(\cC)=2|\ind\cC|-a(\cC)-|\ind_1^+\cC|=1+L-(|\ind_1^+\cC|-1)\).
\end{proof}

We recall translation quivers, used in the next lemma, and their mesh
categories, used in Subsection~\ref{subsec:graded}.
\begin{definition}\label{def:mesh}
A \emph{translation quiver} is a locally
finite quiver without loops together with an injective map \(\tau\), the
translation, from a subset of the vertex set to the vertex set, such that for every
vertex \(X\) in its domain and every vertex \(Y\), the number of arrows
\(Y\to X\) equals the number of arrows \(\tau X\to Y\).
The \(k\)-linear path category of a translation quiver has its vertices as
objects; its morphisms are linear combinations of oriented paths, including
the stationary identity paths, and composition is concatenation.
A \emph{mesh presentation}
is a choice, at each vertex \(X\) with a translate and each vertex \(Y\), of a
bijection between the arrows \(\tau X\to Y\) and the arrows \(Y\to X\); we
write the paired arrows as \(b_{Y,j}:\tau X\to Y\) and
\(a_{Y,j}:Y\to X\). The \emph{mesh relation}
at \(X\) is
\[
 \sum_{Y,j}a_{Y,j}b_{Y,j}=0.
\]
The \emph{mesh category} is the path category modulo these relations.
\end{definition}

A function \(h\) as in Lemma~\ref{lem:height-arrows} is obtained from path
lengths by the following lemma of Ringel and Vossieck.
\begin{lemma}[{\cite[Section~7, Lemma, pp.~234--235]{RV}}]\label{lem:rv-grading}
Let \(\Gamma\) be a translation quiver with a unique source, in which only
finitely many vertices have no translate. Suppose that each vertex without translate
has at most one immediate predecessor, that is, a vertex with an arrow to
it. Then all paths in \(\Gamma\) from the
source to a fixed vertex have the same length.
\end{lemma}

\begin{theorem}\label{thm:poset-magnitude}
Let \(T\) be a finite poset such that \(\TSp\) is representation-finite.
Then its Hom matrix is invertible. Every path of irreducible maps from
\(K_\varnothing\) to \(K_T\) has the same length \(L\), and
\begin{equation}\label{eq:poset-magnitude}
 \chi(\TSp)=1+L-|T|\geq1.
\end{equation}
The following conditions are equivalent:
\begin{enumerate}
\item \(\chi(\TSp)=1\);
\item every indecomposable \(T\)-space has one-dimensional total space;
\item the width of \(T\), the largest size of an antichain, is at most two.
\end{enumerate}
\end{theorem}
\begin{proof}
We use the following results of Ringel and Vossieck on a finite
representation-finite poset \(T\).
\begin{enumerate}[label={\upshape(\alph*)}]
\item The category \(\TSp\) has almost split sequences, that is,
sequences \(X\to Y\to Z\) of \(T\)-spaces whose first map is minimal left
almost split and whose last map is minimal right almost split. They are
exact on total spaces and on all indexed subspaces
\cite[Section~1, p.~219, and Section~7, p.~234]{RV}.
Its Auslander--Reiten quiver is finite, since \(\TSp\) is
representation-finite.
\item The object \(K_\varnothing\) is the unique source and \(K_T\) the
unique sink of the Auslander--Reiten quiver. The vertices without
translate are \(K_\varnothing\) and \(K_{\{s:t\leq s\}}\) for
\(t\in T\), and the minimal right almost split maps ending at them are
monomorphisms. Each of them other than \(K_\varnothing\) is the target of
exactly one arrow. Dually, \(K_T\) has no inverse translate, and every
other vertex without inverse translate is the source of exactly one arrow
\cite[Section~7, pp.~234--235]{RV}.
\item The common length \(L\) of the paths from \(K_\varnothing\) to
\(K_T\) satisfies \(L\geq|T|\) \cite[Section~8, p.~237]{RV}. Moreover,
\(L=|T|\) if and only if every indecomposable has one-dimensional total
space, if and only if \(T\) has width at most two
\cite[Section~8, Proposition, p.~237]{RV}.
\end{enumerate}
We now verify that \(\TSp\) satisfies the hypotheses of
Lemma~\ref{lem:height-arrows}.

A sequence that is
exact on total spaces and on all indexed subspaces is a kernel--cokernel
pair in \(\TSp\): a map into its middle term killed by the last map factors
through the first map on total spaces, and exactness on each indexed
subspace shows that the factorization is a morphism of \(T\)-spaces. Hence
each almost split sequence is a right and a left \(\tau\)-sequence with
monic first map and epic last map. At a vertex without translate, the
monic minimal right almost split map of (b) gives a right
\(\tau\)-sequence with zero first term.

For the dual statement, note that
\((W;W_t)\mapsto(W^*;W_t^\perp)\) is a contravariant equivalence from
\(T\)-spaces to \(T^{\op}\)-spaces. Since \(\Hom(K_\varnothing,W)\) is the
total space of \(W\) and \(\Hom(W,K_T)\) is its dual, the monomorphisms and
epimorphisms of \(T\)-spaces are exactly the maps that are injective,
respectively surjective, on total spaces. The duality exchanges these two
classes.
Applied to \(T^{\op}\), (b) therefore shows that the minimal left almost
split map starting at a vertex without inverse translate is epic. It
gives a left \(\tau\)-sequence with zero last term. Thus \(\TSp\) is a
strict \(\tau\)-category, its \(\tau\)-projectives are the vertices listed
in (b), and \(|\ind_1^+\TSp|=|T|+1\). By
Proposition~\ref{rem:tau-properties}(2) and Lemma~\ref{lem:tau-multiplicity},
the number of arrows \(Y\to X\) equals the number of arrows
\(\tau_{\TSp}^+X\to Y\) whenever \(\tau_{\TSp}^+X\ne0\). Hence the
irreducible quiver of \(\TSp\) with the translation \(\tau_{\TSp}^+\), which
is its Auslander--Reiten quiver, is a translation quiver.

Every indecomposable \(X\not\cong K_\varnothing\) is the target of a path
from \(K_\varnothing\): the space \(\Hom(K_\varnothing,X)\) is the total
space of \(X\), which is nonzero, and a nonzero map \(K_\varnothing\to X\)
is radical, hence a linear combination of paths of irreducible maps by
Lemma~\ref{lem:path-expansion}. By (b) and Lemma~\ref{lem:rv-grading}, the
length of the paths from \(K_\varnothing\) therefore defines a function
\(h\) with \(h(K_\varnothing)=0\) that increases by one along each arrow. If \(\tau_{\TSp}^+X\ne0\),
strictness gives a nonzero middle term, so \(X\) and \(\tau_{\TSp}^+X\) are
joined by paths of length two and \(h(X)=h(\tau_{\TSp}^+X)+2\).
Lemma~\ref{lem:height-arrows} now applies with \(X_0=K_\varnothing\),
\(X_1=K_T\), and \(h(K_T)=L\). It gives \(a(\TSp)=2|\ind\TSp|-L-2\) and
\(\chi(\TSp)=1+L-|T|\). Now (c) gives
\(\chi(\TSp)\geq1\).

(1)\(\Leftrightarrow\)(2): By \eqref{eq:poset-magnitude}, \(\chi(\TSp)=1\) if and only if
\(L=|T|\), and by (c) this holds if and only if every indecomposable has
one-dimensional total space.

(2)\(\Leftrightarrow\)(3): This is part of (c).
\end{proof}

\subsection{Ideal quotients of representation-directed algebras}\label{set:directed-sub}
We now return to modules. The rest of this section compares an algebra with
its quotient by a primitive idempotent.
\begin{definition}\label{def:primitive-deletion}
For a primitive idempotent \(e\) of an algebra \(A\), the algebra
\(A/AeA\) is called the \emph{primitive deletion} of \(e\).
\end{definition}
For a representation-directed algebra and a primitive idempotent, we fix
the following notation for the associated ideal quotient.
\begin{setting}\label{set:directed}
Throughout Subsections~\ref{set:directed-sub}--\ref{subsec:deletion-proof}
and Appendix~\ref{app:structural}, except in statements and examples that
name another algebra, let \(A\) be basic, representation-finite, and
representation-directed. Fix a primitive
idempotent \(e\), and put
\[
 B=A/AeA,\quad \cD=(\mod A)/[\mod B],\quad
 P=eA,\quad E=P/\rad P,\quad I=\Dk(Ae).
\]
Then \(P\) is the projective cover of the simple module \(E\), \(I\) is
its injective envelope, and \(\cD\) is the ideal quotient of
Setting~\ref{set:serre} for this idempotent. Thus \(E\), \(P=P_E\), and
\(I=I_E\) agree with the notation of
Subsection~\ref{subsec:extension-supports}. Put
\(d_X=\dim_k Xe=[X:E]\), where \([X:E]\) is composition multiplicity.
The equality \(\dim_k Xe=[X:E]\) follows by applying the exact functor
\((-)e\) to a composition series.
\end{setting}

The following lemma shows that \(B\) is representation-finite, as assumed
from Subsection~\ref{subsec:new-counts} on, and representation-directed, as
needed for induction on primitive deletions.
\begin{lemma}\label{lem:quotient-closure}
The algebra \(B\) is
representation-finite and representation-directed.
\end{lemma}
\begin{proof}
The inclusion of \(\mod B\) in \(\mod A\) is full
(Setting~\ref{set:serre}), so it preserves indecomposables and their Hom
spaces. Hence \(B\) is representation-finite and directed.
\end{proof}

\subsection{Source, sink, and translation of the quotient}
Lemma~\ref{lem:quotient-evaluation} and
Proposition~\ref{lem:surviving-arrows} already describe the maps and the
indecomposables of \(\cD\). For a directed algebra, we also show that the
indecomposables of \(\cD\) have endomorphism ring \(k\), that \(\cD\) is
directed, and that \(P\) and \(I\) are the unique source and sink of its
irreducible quiver. We identify the indecomposables of \(\cD\) with the
chosen \(A\)-module representatives with \(Xe\ne0\), writing \(X\) for
\(\pi X\) when it occurs as an object of \(\cD\).

\begin{lemma}\label{lem:evaluation}
Evaluation at \(e\) and the quotient functor \(\pi\) have the following properties.
\begin{enumerate}
\item Every indecomposable \(A\)-module has endomorphism ring \(k\).
The category \(\cD\) is Krull--Schmidt and directed. Its indecomposables
 are the images of the indecomposable \(A\)-modules with \(d_X>0\);
 they remain distinct and have endomorphism ring \(k\).
\item For every object \(X\),
\begin{equation}\label{eq:evaluation}
 \cD(P,X)=\Hom_A(P,X)\cong Xe,\qquad
 \cD(X,I)=\Hom_A(X,I)\cong\Dk(Xe).
\end{equation}
The object \(P\) is the unique source and \(I\) the unique sink of the
irreducible quiver of \(\cD\).
\end{enumerate}
\end{lemma}
\begin{proof}
(1) Directedness gives \(\End_A(X)=k\) for every indecomposable:
a nonzero noninvertible endomorphism would be a cycle.
Proposition~\ref{lem:surviving-arrows} identifies the indecomposables
of \(\cD\) and their local endomorphism rings, which are therefore \(k\).
A cycle of nonzero nonisomorphisms in \(\cD\) would lift to one over
\(A\), so \(\cD\) is directed.

(2) The identifications in \eqref{eq:evaluation} hold because
\(\Hom_A(P,\mod B)=0\) and \(\Hom_A(\mod B,I)=0\).
Every indecomposable of \(\cD\) therefore receives a nonzero map from \(P\) and maps
nontrivially to \(I\). An incoming arrow at \(P\), or an outgoing arrow
at \(I\), would give a cycle. Expanding the nonzero maps from \(P\) and to \(I\)
into paths by Lemma~\ref{lem:path-expansion} gives an incoming arrow at
every \(X\ne P\) and an outgoing arrow at every \(X\ne I\), by taking
the last and first arrows of these paths, respectively.
\end{proof}
Part~(2) is also observed in \cite[Section~10]{RV}.

By Proposition~\ref{prop:quotient-sequences}, \(\cD\) is a strict
\(\tau\)-category whose right and left \(\tau\)-sequences at \(X\), in the
notation introduced before Proposition~\ref{rem:tau-properties}, are the
images of the right and left \(\tau\)-sequences of \(\mod A\) ending and
starting at \(X\) (Example~\ref{ex:module-tau}).
\begin{definition}\label{def:boundary}
Write \(\Gamma_{\cD}\) for the irreducible quiver of \(\cD\) equipped
with the translation \(\tau_{\cD}^+\), displayed by dotted arrows
\(X\dashrightarrow\tau_{\cD}^+X\).
\end{definition}
Since \(\cD\) is directed (Lemma~\ref{lem:evaluation}(1)), \(\Gamma_{\cD}\)
has no loops, and \(\tau_{\cD}^+\) is injective by
Proposition~\ref{rem:tau-properties}(3). By
Proposition~\ref{rem:tau-properties}(2) and
Lemma~\ref{lem:tau-multiplicity}, the number of arrows \(Y\to X\) equals the
number of arrows \(\tau_{\cD}^+X\to Y\) whenever \(\tau_{\cD}^+X\ne0\). Thus
\(\Gamma_{\cD}\) is a translation quiver (Definition~\ref{def:mesh}).

\begin{lemma}\label{rem:iyama-boundary}
Let \(X\) be an indecomposable \(A\)-module with \(Xe\ne0\).
Its image is \(\tau\)-projective in \(\cD\) precisely when \(X\) is
projective over \(A\) or \(\tau_A X\) belongs to \(\mod B\). Dually, it is
\(\tau\)-injective precisely when \(X\) is injective over \(A\) or
\(\tau_A^-X\) belongs to \(\mod B\).
\end{lemma}
\begin{proof}
By Proposition~\ref{prop:quotient-sequences},
\(\tau_{\cD}^+(\pi X)\cong\pi(\tau_A X)\), and \(\pi\) kills exactly the
modules in \(\mod B\) (Proposition~\ref{lem:surviving-arrows}). The dual
statement follows in the same way.
\end{proof}

\subsection{Composition multiplicities at \texorpdfstring{\(\tau\)}{tau}-projectives and \texorpdfstring{\(\tau\)}{tau}-injectives}\label{subsec:boundary}
We next bound the multiplicity \(d_X\) of \(E\) at the \(\tau\)-projectives
and \(\tau\)-injectives of \(\cD\). The bound shows that an almost split
sequence with exactly one end term in \(\mod B\) has exactly one middle
summand outside it, which removes the term \(\delta\) in
Theorem~\ref{thm:serre-magnitude}.

\begin{proposition}[{\cite[Section~10, Corollary, p.~243]{RV}}]\label{prop:boundary-one}
In Setting~\ref{set:directed}, the following hold.
\begin{enumerate}
\item Every \(\tau\)-projective and \(\tau\)-injective of \(\cD\) has
\(d_X=1\).
\item For a nonprojective indecomposable \(A\)-module \(Y\),
\begin{equation}\label{eq:zero-end}
 (\tau_A Y)e=0\ \Longrightarrow\ d_Y\leq1,\qquad
 Ye=0\ \Longrightarrow\ d_{\tau_A Y}\leq1.
\end{equation}
\end{enumerate}
\end{proposition}
\begin{proof}
Ringel and Vossieck show that over a representation-directed algebra,
every indecomposable projective or injective module \(Z\) has
\(d_Z\leq1\), and \(|d_Z-d_{\tau_A Z}|\leq1\) for every nonprojective
indecomposable \(Z\) \cite[Section~10, Corollary, p.~243]{RV}; they
state that it is due to v.~H\"ohne. Part~(2) is the case \(Z=Y\) of this bound.
For part~(1), a \(\tau\)-projective \(X\) of \(\cD\) is projective over
\(A\) or satisfies \((\tau_A X)e=0\), by Lemma~\ref{rem:iyama-boundary}. In
both cases \(d_X\leq1\), and \(d_X\geq1\) since \(X\notin\mod B\). The
\(\tau\)-injectives are dual.
\end{proof}
Appendix~\ref{app:multiplicity} gives a proof that does not use \cite{RV}.

In the following example the multiplicities of
Proposition~\ref{prop:boundary-one} can be computed directly; it is used
again for the poset construction in Example~\ref{ex:a3-poset}.
\begin{example}\label{ex:a3}
Let \(Q\) be \(1\to2\to3\) and \(A=(kQ)^\op\). Write
\(M_{ij}\) for the interval representation on vertices \(i,\ldots,j\),
with identity maps between nonzero coordinates, and \(S_i=M_{ii}\).
The Auslander--Reiten quiver of \(A\) is
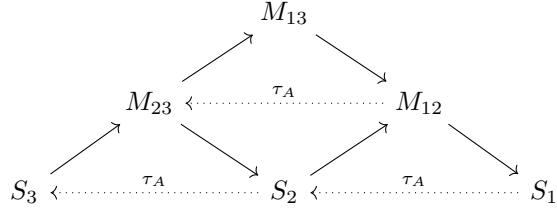
\[
\begin{tikzcd}
 &&M_{13}\arrow[dr]&&\\
 &M_{23}\arrow[ur]\arrow[dr]&&M_{12}\arrow[dr]
   \arrow[ll,dotted,"\tau_A"']&\\
 S_3\arrow[ur]&&S_2\arrow[ur]\arrow[ll,dotted,"\tau_A"']&&
 S_1\arrow[ll,dotted,"\tau_A"']
\end{tikzcd}
\]
Take \(e\) to be the vertex idempotent at \(2\), so \(E=S_2\).
Then \(B\cong k\times k\), whose indecomposables
\(S_1,S_3\) become zero in \(\cD\). The quotient has translation quiver
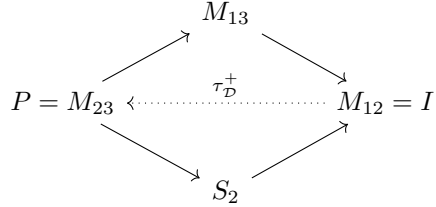
\[
\begin{tikzcd}
 &M_{13}\arrow[dr]&\\
 P=M_{23}\arrow[ur]\arrow[dr]&&M_{12}=I\arrow[ll,dotted,"\tau_{\cD}^+"']\\
 &S_2\arrow[ur]&
\end{tikzcd}
\]
The only nonzero translate is \(\tau_{\cD}^+M_{12}=M_{23}\). In particular,
the right \(\tau\)-sequence at \(S_2\) has first term zero because its
first term over \(A\) is \(S_3\). Hence
\(\ind_1^+\cD=\{M_{23},M_{13},S_2\}\). Each has dimension one at
vertex 2, as does \(I=M_{12}\).
\end{example}
\subsection{Identification with poset representations}
We now construct the poset \(T\) and the functor identifying \(\cD\) with
the category of \(T\)-spaces. The functor \(G\) below is the functor of
Ringel and Vossieck \cite[Section~1, Theorem~3, p.~221, and
Section~7]{RV}, written in terms of subspaces; Appendix~\ref{app:poset}
proves the equivalence by a different route.
\begin{construction}\label{con:poset}
Index the \(\tau\)-projectives other than \(P\) by a set \(T\), choosing
indecomposable \(A\)-module representatives \(Q_t\). Define
\[
 s\leq t\quad\Longleftrightarrow\quad\cD(Q_t,Q_s)\ne0.
\]
This order is chosen so that \(Q_t\) corresponds to
\(K_{\{s:t\leq s\}}\) (Proposition~\ref{prop:poset-equivalence}).
Let \(F(X)=Xe\) be the evaluation functor of
Lemma~\ref{lem:quotient-evaluation}. Since \(eAe\cong\End_A(P)=k\) by
Lemma~\ref{lem:evaluation}(1),(2), identify \(F(P)=eAe\) with \(k\), using
\(e\) as basis.
By Proposition~\ref{prop:boundary-one}(1), each \(F(Q_t)\) is one-dimensional.
Choose a basis vector of each \(F(Q_t)\), and let \(u_t\in\cD(P,Q_t)\) be
the map corresponding to it under \eqref{eq:evaluation}, so that its
evaluation sends \(1\) to that basis vector. Identify
\(\cD(P,X)\) with \(F(X)\) by \eqref{eq:evaluation}, and set
\begin{equation}\label{eq:poset-functor}
 F_t(X)=\operatorname{Im}\bigl(\cD(Q_t,X)\longrightarrow\cD(P,X),
                  \ f\longmapsto fu_t\bigr),\qquad
 G(X)=(F(X);F_t(X)).
\end{equation}
On morphisms, \(G\) is evaluation at \(e\).
The subspace \(F_t(X)\) does not change if \(u_t\) is rescaled, and \(G\)
on morphisms does not depend on the chosen bases. The bases are used only
in Appendix~\ref{app:poset}, to realize an arbitrary \(T\)-space by a
module.
\end{construction}

\begin{proposition}[{\cite[Section~10, Theorem, p.~243]{RV}}]\label{prop:poset-equivalence}
Construction~\ref{con:poset} defines a finite poset and a \(k\)-linear
equivalence \(G:\cD\xrightarrow{\sim}\TSp\). Moreover,
\[
 G(P)=K_\varnothing,\qquad G(Q_t)=K_{\{s:t\leq s\}},\qquad G(I)=K_T,
 \qquad |T|=|\ind_1^+\cD|-1.
\]
\end{proposition}
We compute the equivalence on the four objects of Example~\ref{ex:a3}.
\begin{example}\label{ex:a3-poset}
For Example~\ref{ex:a3}, let \(Q_{t_1}=M_{13}\) and \(Q_{t_2}=S_2\).
They have no nonzero maps between them in \(\cD\), so \(T\) is a
two-element antichain. Evaluation and the two indexed subspaces give
\[
\begin{array}{c|ccc|c}
 X&F(X)&F_{t_1}(X)&F_{t_2}(X)&G(X)\\\hline
 M_{23}&k&0&0&K_\varnothing\\
 M_{13}&k&k&0&K_{\{t_1\}}\\
 S_2&k&0&k&K_{\{t_2\}}\\
 M_{12}&k&k&k&K_T
\end{array}
\]
Each arrow of the translation quiver of \(\cD\) in Example~\ref{ex:a3}
is the identity on total spaces.
The paths of irreducible maps from \(K_\varnothing\) to \(K_T\) have length
\(L=2=|T|\), so \(\chi(\cD)=1\) by Theorem~\ref{thm:poset-magnitude}.
\end{example}

\begin{corollary}\label{prop:intrinsic}
One has \(\chi(\cD)\geq1\), and the following are equivalent.
\begin{enumerate}
\item \(\chi(\cD)=1\).
\item \(d_X=1\) for every \(X\in\ind\cD\).
\item \(T\) has width at most two.
\end{enumerate}
\end{corollary}
\begin{proof}
The equivalence of Proposition~\ref{prop:poset-equivalence} sends \(X\) to
a \(T\)-space with total space \(Xe\), and \(\TSp\) is
representation-finite because \(\cD\) is. Thus \(\chi(\cD)=\chi(\TSp)\), and
conditions (1)--(3) are conditions (1)--(3) of
Theorem~\ref{thm:poset-magnitude}, which proves the corollary.
\end{proof}
Appendix~\ref{app:inequality} proves this corollary directly, without
Theorem~\ref{thm:poset-magnitude}.

In the following \(D_4\) example, \(\chi(\cD)>1\), in contrast with
Example~\ref{ex:a3-poset}.
\begin{example}\label{ex:d4-subspaces}
Let \(Q\) be the Dynkin quiver of type \(D_4\) with arrows
\(1\to0\), \(2\to0\), \(3\to0\), and put
\(A=(kQ)^\op\). Take \(e\) at vertex 0. The three
\(\tau\)-projectives other than \(P=S_0\) are the projectives supported
on \(\{0,i\}\), so \(T\) is a three-element antichain. We identify the
element of \(T\) indexing the projective supported on \(\{0,i\}\) with
\(i\), so that \(T=\{1,2,3\}\).

Its representations are triples of subspaces of a vector space. Besides
the eight one-dimensional representations \(K_U\), for \(U\subseteq T\),
there is an indecomposable \(Z\) with total space \(k^2\) and indexed lines
\(k(1,0)\), \(k(0,1)\), and \(k(1,1)\).
The \(A\)-module \(K_U\) has space \(k\) at 0 and at the vertices in
\(U\), with identity maps between nonzero spaces. The module \(Z\)
has space \(k^2\) at 0, and its three arrow maps are the inclusions
of the lines \(k(1,0)\), \(k(0,1)\), and \(k(1,1)\). The functor \(G\) of
Proposition~\ref{prop:poset-equivalence} sends these modules to the
\(T\)-spaces with the same names.

Together with the simples \(S_1,S_2,S_3\), these are the indecomposable
\(A\)-modules, and knitting gives the Auslander--Reiten quiver of \(A\):
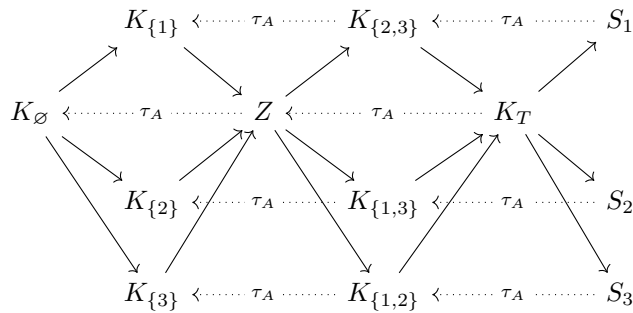
\[
\begin{tikzcd}[row sep=1.7em,column sep=1.9em]
 & K_{\{1\}} \arrow[dr] && K_{\{2,3\}} \arrow[dr]
   \arrow[ll,dotted,"\tau_A" description] && S_1
   \arrow[ll,dotted,"\tau_A" description] \\
 K_\varnothing \arrow[ur] \arrow[dr] \arrow[ddr]
 && Z \arrow[ur] \arrow[dr] \arrow[ddr]
   \arrow[ll,dotted,"\tau_A" description]
 && K_T \arrow[ur] \arrow[dr] \arrow[ddr]
   \arrow[ll,dotted,"\tau_A" description] & \\
 & K_{\{2\}} \arrow[ur] && K_{\{1,3\}} \arrow[ur]
   \arrow[ll,dotted,"\tau_A" description] && S_2
   \arrow[ll,dotted,"\tau_A" description] \\
 & K_{\{3\}} \arrow[uur] && K_{\{1,2\}} \arrow[uur]
   \arrow[ll,dotted,"\tau_A" description] && S_3
   \arrow[ll,dotted,"\tau_A" description]
\end{tikzcd}
\]
The projectives are \(K_\varnothing,K_{\{1\}},K_{\{2\}},K_{\{3\}}\).
The arrow \(K_{\{i\}}\to Z\) is the inclusion of the \(i\)-th indexed line,
the arrow \(Z\to K_{T\setminus\{i\}}\) is the quotient map with kernel
that line, and the other arrows between the modules \(K_U\) are
identities on total spaces.
The almost split sequences ending at \(Z\), \(K_{\{i,j\}}\), \(K_T\),
and \(S_i\) have middle terms \(\bigoplus_iK_{\{i\}}\), \(Z\),
\(\bigoplus_{|U|=2}K_U\), and \(K_T\), respectively.

The modules \(S_1,S_2,S_3\) become zero in \(\cD\). By
Propositions~\ref{lem:surviving-arrows} and~\ref{prop:quotient-sequences},
the remaining arrows and the images of the translations give the
translation quiver \(\Gamma_{\cD}\):
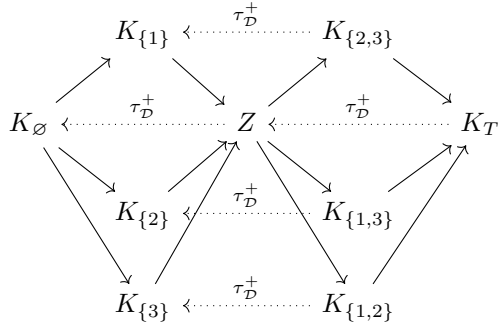
\[
\begin{tikzcd}[column sep=1.6em,row sep=1.7em]
 &K_{\{1\}}\arrow[dr]&&K_{\{2,3\}}\arrow[dr]
   \arrow[ll,dotted,"\tau_{\cD}^+"']&\\
 K_\varnothing\arrow[ur]\arrow[dr]\arrow[ddr]&&
 Z\arrow[ur]\arrow[dr]\arrow[ddr]\arrow[ll,dotted,"\tau_{\cD}^+"']&&
 K_T\arrow[ll,dotted,"\tau_{\cD}^+"']\\
 &K_{\{2\}}\arrow[ur]&&K_{\{1,3\}}\arrow[ur]
   \arrow[ll,dotted,"\tau_{\cD}^+"']&\\
 &K_{\{3\}}\arrow[uur]&&K_{\{1,2\}}\arrow[uur]
   \arrow[ll,dotted,"\tau_{\cD}^+"']&
\end{tikzcd}
\]
The quotient has nine indecomposables, twelve arrows, and four
\(\tau\)-projectives. Thus \eqref{eq:categorical-magnitude} gives
\[
 \chi(\cD)=2\cdot9-12-4=2,\qquad L=4>|T|=3.
\]
\end{example}
\subsection{Magnitude under deletion of a simple module}\label{subsec:deletion-proof}
We now apply Theorem~\ref{thm:serre-magnitude} and
Corollary~\ref{cor:directed-compensation} to the Serre subcategory
\(\mod B\).

\begin{proposition}\label{thm:directed-deletion}
In Setting~\ref{set:directed}, with \(c,r\) as in
Definition~\ref{def:new-counts}, one has
\begin{equation}\label{eq:directed-deletion}
 \sigma(A)-\sigma(B)=\chi(\cD)-1+c-r\geq0.
\end{equation}
If equality holds in \(\sigma(A)\geq\sigma(B)\), then
\(\dim_k Xe\leq1\) for every indecomposable \(A\)-module \(X\).
\end{proposition}

\begin{proof}
By Lemma~\ref{lem:quotient-closure}, \(\mod B\) is a
representation-finite Serre subcategory, so Section~\ref{sec:quotient}
applies. In an almost split sequence with exactly one end term in \(\mod B\),
Proposition~\ref{prop:boundary-one}(2) gives \(d_X\leq1\) for the other
end term \(X\), and \(d_X\geq1\) since \(X\notin\mod B\). The sequence
stays exact after applying \((-)e\), and its end term in \(\mod B\)
contributes zero. So \(d_X=1\) is the sum of \(d_Z\geq1\) over the middle
summands \(Z\) outside \(\mod B\), counted with multiplicity. Hence there
is exactly one middle summand outside \(\mod B\), with multiplicity one. Thus \(\delta=0\) in
Theorem~\ref{thm:serre-magnitude}. Since \(|A|=|B|+1\), its identity becomes
\[
 \sigma(A)-\sigma(B)=\chi(\cD)-1+c-r.
\]
Corollaries~\ref{prop:intrinsic} and~\ref{cor:directed-compensation}
make the two terms nonnegative. Equality implies \(\chi(\cD)=1\), so
\(d_X=1\) for every \(X\in\ind\cD\) by
Corollary~\ref{prop:intrinsic}\((1)\Rightarrow(2)\); the other
indecomposable \(A\)-modules have \(d_X=0\).
\end{proof}

Deleting a leaf in Example~\ref{ex:d4-subspaces} gives a strict inequality
\(r<c\). We compute \(c\) and \(r\) explicitly.
\begin{example}\label{ex:d4-boundary}
Keep \(A=(kQ)^\op\) with arrows \(i\to0\), \(i=1,2,3\), from
Example~\ref{ex:d4-subspaces}. Write \(e_i\) for the idempotent at vertex
\(i\), and now take \(E=S_3\) and \(B=A/Ae_3A\).
Keep the notation \(K_U,Z\) from that example.
The indecomposable \(B\)-modules are \(K_\varnothing,K_{\{1\}},K_{\{2\}},
K_{\{1,2\}},S_1,S_2\). The Auslander--Reiten quiver \(\Gamma_B\) is
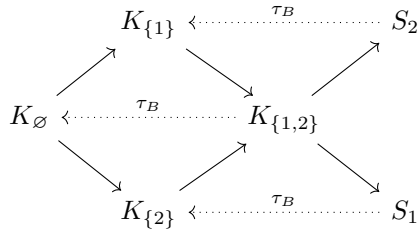
\[
\begin{tikzcd}[column sep=1.8em,row sep=1.8em]
 &K_{\{1\}}\arrow[dr]&&S_2\arrow[ll,dotted,"\tau_B"']\\
 K_\varnothing\arrow[ur]\arrow[dr]&&K_{\{1,2\}}\arrow[ur]\arrow[dr]
   \arrow[ll,dotted,"\tau_B"']&\\
 &K_{\{2\}}\arrow[ur]&&S_1\arrow[ll,dotted,"\tau_B"']
\end{tikzcd}
\]
The two arrows out of \(K_\varnothing\) are irreducible over \(A\).
The other four are not. Their factorizations in \(\mod A\) are shown
in the following diagram of maps:
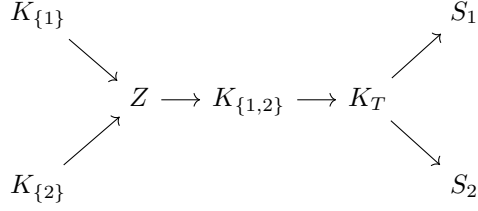
\[
\begin{tikzcd}[column sep=1.5em,row sep=1.5em]
 K_{\{1\}}\arrow[dr]&&&&S_1\\
 &Z\arrow[r]&K_{\{1,2\}}\arrow[r]&K_T\arrow[ur]\arrow[dr]&\\
 K_{\{2\}}\arrow[ur]&&&&S_2
\end{tikzcd}
\]
Each map in this diagram is irreducible in \(\mod A\). The two composites
through \(Z\) and the two through \(K_T\) give the other four
arrows of \(\Gamma_B\): both intermediate modules have nonzero space at
vertex 3, so they are not \(B\)-modules. Each of these four arrows
contributes one to \(c\), and there are no other arrows in \(\Gamma_B\).
Thus \(c=4\).

The three almost split sequences of \(\mod B\) are
\[
\begin{gathered}
0\to K_\varnothing\to K_{\{1\}}\oplus K_{\{2\}}
  \to K_{\{1,2\}}\to0,\\
0\to K_{\{2\}}\to K_{\{1,2\}}\to S_1\to0,\qquad
0\to K_{\{1\}}\to K_{\{1,2\}}\to S_2\to0.
\end{gathered}
\]
None is almost split over \(A\): the almost split sequences over
\(A\) with the same last terms start at \(K_{\{3\}},K_{\{2,3\}},K_{\{1,3\}}\), respectively.
Hence \(r=3\), and \(c-r=1\).

For the functors of Subsection~\ref{subsec:extension-supports}, the
projective resolution \(0\to S_0\to K_{\{3\}}\to S_3\to0\) gives
\(\mathsf H_E(X)=\Ext_A^1(S_3,X)\cong Xe_0\) for \(X\in\mod B\).
The simple \(S_3\) is injective, so \(\mathsf G_E=0\).
Theorem~\ref{prop:compensation} therefore assigns to each of the three
sequences a pair \((W,Y)\) with \(c_{WY}>0\), where \(Y\) is its last term. The sequence ending at \(K_{\{1,2\}}\)
has two possible choices, from \(K_{\{1\}}\) and \(K_{\{2\}}\);
the other two sequences have one each. This displays the strict inequality
\(r=3<4=c\).
Example~\ref{ex:d4-deletion} combines these counts with those of
\(\cD\) to check \eqref{eq:directed-deletion}.
\end{example}
The two deletions in our \(D_4\) example distinguish the two nonnegative
terms in \eqref{eq:directed-deletion}.
\begin{example}\label{ex:d4-deletion}
The Auslander--Reiten quiver of \(\mod A\), displayed in
Example~\ref{ex:d4-subspaces}, has twelve indecomposables and fifteen
arrows. Thus, by Lemma~\ref{lem:count},
\[
 |\ind A|=12,\quad |A|=4,\quad a(A)=15,\quad
 \chi(\mod A)=5,\quad\sigma(A)=1.
\]
Deleting vertex 0 gives \(B_0\cong k^3\). Deleting vertex 3 gives the
path algebra \(B_3\) of Example~\ref{ex:d4-boundary}.
The complete counts are
\[
\begin{array}{c|ccc|ccc|c}
 e&|\ind B|&|B|&a(B)&\chi(\cD)-1&c&r&\sigma(A)-\sigma(B)\\\hline
 e_0&3&3&0&1&0&0&1\\
 e_3&6&3&6&0&4&3&1
\end{array}
\]
For the second row, the associated translation quiver \(\Gamma_{\cD}\) is
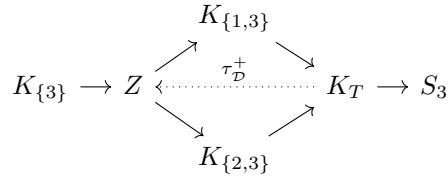
\[
\begin{tikzcd}[column sep=small,row sep=small]
 &&K_{\{1,3\}}\arrow[dr]&&\\
 K_{\{3\}}\arrow[r]&Z\arrow[ur]\arrow[dr]&&
 K_T\arrow[r]\arrow[ll,dotted,"\tau_{\cD}^+"']&S_3\\
 &&K_{\{2,3\}}\arrow[ur]&&
\end{tikzcd}
\]
Thus it has six indecomposables and six solid arrows. Only
\(K_T\) has a nonzero \(\tau_{\cD}^+\)-translate, namely \(Z\).
Hence there are five \(\tau\)-projectives, and
\eqref{eq:categorical-magnitude} gives \(\chi(\cD)=2\cdot6-6-5=1\). The strict increase comes entirely from
\(\chi(\cD)-1\) in the first row and from \(c-r\) in the second.
\end{example}
\subsection{Magnitude of representation-directed algebras}\label{subsec:directed-case}
Applying Proposition~\ref{thm:directed-deletion} to one simple module at a
time bounds the surplus of a representation-directed algebra from below,
and zero surplus leads to biserial algebras. We recall the classes of
algebras that enter the proof of Theorem~A(2).
\begin{definition}\label{def:special-biserial}
A basic algebra is \emph{special biserial} if it has a presentation
\(kQ/\mathfrak r\), with \(R_Q^m\subseteq\mathfrak r\subseteq R_Q^2\)
for the arrow ideal \(R_Q\) and some \(m\geq2\), satisfying two conditions:
at most two arrows enter and at most two leave each vertex; and for each
arrow \(\alpha\), there is at most one arrow \(\beta\) for which
\(\beta\alpha\) is defined and does not belong to \(\mathfrak r\),
and at most one arrow \(\gamma\) with the same property for
\(\alpha\gamma\). These conditions are symmetric under replacing \(Q\) by
its opposite quiver.
\end{definition}

\begin{definition}\label{def:biserial}
A module is \emph{uniserial} if its submodules are totally ordered by
inclusion. An algebra is \emph{biserial} if the radical of every
indecomposable projective, on either side, is a sum of two uniserial
submodules whose intersection is simple or zero. A module over a basic
algebra is \emph{thin} if each simple occurs in it with multiplicity
at most one, equivalently if \(\dim_k Xe\leq1\) at every primitive
idempotent in a complete orthogonal set.
\end{definition}
If a representation-directed algebra has zero surplus, then all its
indecomposable modules are thin (Corollary~\ref{cor:directed} below).
Two results show that such an algebra is special biserial. The first is
due to Pogorza{\l}y and Skowro\'nski.

\begin{proposition}[{\cite[Proposition~1]{PS}}]\label{prop:thin-biserial}
If every indecomposable module over a basic finite-dimensional algebra is
thin, then the algebra is representation-finite and biserial.
\end{proposition}

The second is due to Skowro\'nski and Waschb\"usch.
\begin{theorem}[{\cite[p.~173 and Lemma~2]{SW}; see also \cite[Theorem~2]{PS}}]\label{thm:biserial-special}
Every basic representation-finite biserial algebra is special biserial.
\end{theorem}

\begin{corollary}\label{cor:directed}
For a representation-finite representation-directed algebra \(A\),
\(\sigma(A)\geq0\). If \(A\) is basic and \(\sigma(A)=0\), then every
indecomposable \(A\)-module is thin, and \(A\) is biserial and hence
special biserial.
\end{corollary}
\begin{proof}
By Lemma~\ref{lem:invariance} we may assume that \(A\) is basic. The
algebras obtained by successive primitive deletions remain basic, and by Lemma~\ref{lem:quotient-closure}
they remain representation-finite and representation-directed. They end at
the zero algebra, whose surplus is zero. Proposition~\ref{thm:directed-deletion}
therefore proves nonnegativity. If \(\sigma(A)=0\), delete any primitive
idempotent \(e\) first. Since \(\sigma(B)\geq0\) for \(B=A/AeA\), we have
\(\sigma(A)=\sigma(B)=0\), so \(\dim_k Xe\leq1\) for every indecomposable \(X\).
As \(e\) was arbitrary, all indecomposables are thin.
Proposition~\ref{prop:thin-biserial} shows that \(A\) is biserial, and
Theorem~\ref{thm:biserial-special} that it is special biserial.
\end{proof}
\section{A lower bound for magnitude via graded interval algebras}\label{sec:general}
Corollary~\ref{cor:directed} proves \(\chi(\mod A)\geq|A|\) only for
representation-directed algebras. To remove this hypothesis, we compare a
graded algebra with its interval algebras, whose modules are graded modules
supported in finite degree intervals. We first show that the numbers of
indecomposables and of arrows of the interval algebras grow linearly with
the length of the interval, with slopes the numbers of indecomposables and
of arrows of the algebra. A positive grading of
the Hom spaces between indecomposable modules (Subsection~\ref{subsec:graded-hom})
then makes the interval algebras representation-directed, and standard forms have such a grading.

\subsection{Graded modules supported in finite intervals}
We first fix conventions for graded modules and define interval algebras.
\begin{definition}\label{def:graded-modules}
Let \(R=\bigoplus_{d\geq0}R^{(d)}\) be a nonnegatively graded algebra.
A graded right \(R\)-module is a decomposition
\(M=\bigoplus_{n\in\mathbb Z}M_n\) with
\(M_nR^{(d)}\subseteq M_{n+d}\). Its support consists of the integers
\(n\) with \(M_n\ne0\). Write \(\operatorname{grmod}R\) for the category
of finite-dimensional graded modules and maps of degree zero, and
\(\Hom_{\mathrm{gr}}\) for its morphism spaces.
The shift \(M\langle s\rangle\) is defined by
\((M\langle s\rangle)_n=M_{n-s}\). A homogeneous map of degree \(d\)
sends \(M_n\) to \(N_{n+d}\); it has degree zero as a map
\(M\langle s\rangle\to N\langle t\rangle\) when \(d=s-t\).
\end{definition}

\begin{construction}\label{con:interval}
For \(m\geq0\), define the \emph{interval algebra} \(R_m\) to be the
upper triangular block algebra with blocks indexed by \(0,\ldots,m\),
\begin{equation}\label{eq:interval-algebra}
 (R_m)_{rs}=\begin{cases}R^{(s-r)}&r\leq s,\\0&r>s.\end{cases}
\end{equation}
Multiplication of blocks is induced by multiplication in \(R\).
Thus its diagonal blocks are copies of \(R^{(0)}\).
\end{construction}
\begin{remark}
Viewed as a category with objects \(0,\ldots,m\), the algebra \(R_m\) is
the full subcategory on these objects of the smash product of \(R\) with
\(\mathbb Z\) in the sense of Cibils and Marcos
\cite[Definition~3.1]{CM}. The equivalence of \(\mod R_m\) with graded
modules in Proposition~\ref{prop:interval-modules} below is a case of
\cite[Theorem~6.6]{GG1}.
\end{remark}

\begin{proposition}\label{prop:interval-modules}
The category \(\mod R_m\) is equivalent to the full subcategory of
\(\operatorname{grmod}R\) supported in \([0,m]\), and
\(|R_m|=(m+1)|R|\).
If \(R\) is representation-finite, every indecomposable \(R\)-module
has a graded lift, that is, a graded module whose underlying module is
isomorphic to it; this lift is unique up to shift, and \(R_m\) is
representation-finite.
\end{proposition}
\begin{proof}
The representation-finiteness of \(R_m\) is also
\cite[Proposition~1.1]{GG2}; we include the short arguments.
A right \(R_m\)-module has spaces \(M_r\) at the diagonal idempotents,
with multiplication maps \(M_rR^{(s-r)}\to M_s\). These give a graded
\(R\)-module by setting its other components to zero. Associativity
holds across the interval boundary because multiplication only increases
degree: if the initial and final degrees lie in \([0,m]\), so does
every intermediate degree. This construction and restriction to the
interval are inverse on modules and maps.
The ideal of elements of positive degree in \(R\) is nilpotent, so
\(|R|=|R^{(0)}|\). The strictly upper triangular ideal of \(R_m\) is
also nilpotent, and its diagonal quotient is \((R^{(0)})^{m+1}\),
which proves \(|R_m|=(m+1)|R|\).

Gordon--Green prove that a finite-dimensional graded module is
indecomposable if and only if its underlying module is indecomposable,
and that two graded lifts of an indecomposable differ by shift
\cite[Theorems~3.2 and~4.1]{GG1}. They also prove that every module over a
representation-finite graded algebra is gradable
\cite[Theorem~4.3]{GG2}. Choose one lift of each indecomposable.
Only finitely many of its shifts are supported in \([0,m]\), so the equivalence
proves representation-finiteness of \(R_m\).
\end{proof}

\begin{theorem}\label{prop:interval-count}
Let \(R\) be a nonnegatively graded representation-finite algebra.
Choose a graded lift \(\widetilde X\) of each \(X\in\ind R\), normalized so that the least
\(n\) with \(\widetilde X_n\ne0\) is zero, and let \(w_X\) be the greatest such
\(n\).
\begin{enumerate}
\item For \(m\geq\max_X w_X\),
\begin{equation}\label{eq:interval-objects}
 |\ind R_m|=\sum_{X\in\ind R}(m+1-w_X).
\end{equation}
\item As \(m\to\infty\),
\begin{align}
 a(R_m)&=(m+1)a(R)+O(1),\label{eq:interval-arrows}\\
 \sigma(R_m)&=(m+1)\sigma(R)+O(1).\label{eq:interval-surplus}
\end{align}
\end{enumerate}
\end{theorem}
\begin{proof}
(1) The shifts of \(\widetilde X\) in the interval are precisely
\(\widetilde X\langle t\rangle\) with \(0\leq t\leq m-w_X\).
Proposition~\ref{prop:interval-modules} proves \eqref{eq:interval-objects}.

(2) Choose \(b\geq1\) bounding all \(w_X\). For \(X,Y\in\ind R\), every
map \(X\to Y\) is a sum of maps homogeneous with respect to the gradings
of \(\widetilde X\) and \(\widetilde Y\), whose degrees lie in \([-b,b]\)
because both lifts are supported in \([0,b]\).

We first compare irreducible maps in \(\mod R\) and in
\(\operatorname{grmod}R\). Fix \(t\in\mathbb Z\). A homogeneous map
\(X\to Y\) of degree \(d\) is a map of degree zero
\(\widetilde X\langle t+d\rangle\to \widetilde Y\langle t\rangle\), so decomposing maps into
homogeneous components gives (cf. \cite[Lemma~2.1]{GG1})
\begin{equation}\label{eq:graded-hom-sum}
 \Hom_R(X,Y)=\bigoplus_{s\in\mathbb Z}
 \Hom_{\mathrm{gr}}(\widetilde X\langle s\rangle,\widetilde Y\langle t\rangle).
\end{equation}
For distinct underlying indecomposables every map is radical. For an
endomorphism, a homogeneous component of nonzero degree is nilpotent, while
its component of degree zero is radical exactly when its image in the residue
field \(k\) is zero. Hence the radical of \(\mod R\) between the underlying
modules of graded lifts is a graded subspace, that is, it contains the
homogeneous components of its elements. A map of degree zero between graded
indecomposables is an isomorphism in \(\operatorname{grmod}R\) exactly
when its underlying map is an isomorphism, so the radical of
\(\operatorname{grmod}R\) consists of the maps of degree zero in the radical
of \(\mod R\). Therefore, under \eqref{eq:graded-hom-sum}, the space
\(\rad_R(X,Y)\) is the direct sum of the spaces
\(\rad_{\operatorname{grmod}R}(\widetilde X\langle s\rangle,\widetilde Y\langle t\rangle)\). A
radical composite factors through a finite direct sum of indecomposable
\(R\)-modules, each of which has a graded lift by
Proposition~\ref{prop:interval-modules}. Composites of homogeneous maps are
homogeneous, with degrees adding, so decomposing both factors into
homogeneous parts shows that the same holds for the radical squares
\(\rad_R^2\) and \(\rad_{\operatorname{grmod}R}^2\).
Consequently (cf. \cite[Theorem~3.5]{GG2})
\begin{equation}\label{eq:graded-irreducible}
 \bigoplus_{s\in\mathbb Z}
 \Irr_{\operatorname{grmod}R}(\widetilde X\langle s\rangle,\widetilde Y\langle t\rangle)
 \cong\Irr_R(X,Y).
\end{equation}

We now restrict to the interval \([0,m]\) with \(m\geq3b\). Fix \(Y\) and a
shift \(t\) with \(b\leq t\leq m-2b\). By \eqref{eq:graded-hom-sum} and
the bound on degrees, a nonzero map \(\widetilde Z\langle u\rangle\to \widetilde Y\langle t\rangle\) has
\(u\in[t-b,t+b]\), so its source is supported in
\([t-b,t+2b]\subseteq[0,m]\). In a radical factorization of a map into
\(\widetilde Y\langle t\rangle\), we may omit the composites with a zero factor;
each remaining intermediate object has a nonzero map to
\(\widetilde Y\langle t\rangle\), so it is also supported in \([0,m]\). Since
\(\mod R_m\) is a full subcategory of \(\operatorname{grmod}R\), its
radical and radical square at the target \(\widetilde Y\langle t\rangle\) therefore
agree with those of \(\operatorname{grmod}R\). By
\eqref{eq:graded-irreducible}, the indegree of \(\widetilde Y\langle t\rangle\) in
\(\mod R_m\) is \(\operatorname{indeg}_R(Y)\). There are \(m-3b+1\) such
shifts for each \(Y\), and \(\sum_Y\operatorname{indeg}_R(Y)=a(R)\), so
these targets contribute \((m-3b+1)a(R)\) arrows.

For each \(Y\), at most \(3b\) of the shifts \(t\) with
\(\widetilde Y\langle t\rangle\) supported in \([0,m]\) lie outside \([b,m-2b]\).
At such a target, every irreducible space is a subquotient of a Hom space,
and by \eqref{eq:graded-hom-sum} the total incoming Hom dimension is at
most \(\sum_X\dim_k\Hom_R(X,Y)\), independently of \(t\) and \(m\). These
targets therefore contribute \(O(1)\) arrows, which proves
\eqref{eq:interval-arrows}. Finally, Lemma~\ref{lem:count} gives
\(\sigma(R_m)=2(|\ind R_m|-|R_m|)-a(R_m)\). Substituting
\eqref{eq:interval-objects}, \(|R_m|=(m+1)|R|\) from
Proposition~\ref{prop:interval-modules}, and
\eqref{eq:interval-arrows} proves \eqref{eq:interval-surplus}.
\end{proof}

\subsection{Positive gradings of Hom spaces}\label{subsec:graded-hom}
Interval algebras need not be representation-directed. The following
grading of the Hom spaces between indecomposable modules ensures that
they are.
\begin{definition}\label{def:positive-hom-grading}
A \emph{positive Hom grading} on \(\ind R\) is a decomposition
\[
 \Hom_R(X,Y)=\bigoplus_{d\geq0}\Hom^{(d)}(X,Y),\qquad
 \Hom^{(0)}(X,Y)=\begin{cases}k1_X&X=Y,\\0&X\ne Y,\end{cases}
\]
such that composition adds degrees.
\end{definition}

\begin{proposition}\label{prop:graded-classification}
Let \(R\) be basic and representation-finite with a positive Hom grading.
\begin{enumerate}
\item The grading induces a nonnegative grading on \(R=\End_R(R)\) and
graded lifts
\[
 M_X=\Hom_R(R,X),\qquad
 (M_X)_d=\bigoplus_i\Hom^{(d)}(e_iR,X).
\]
\item For these lifts,
\begin{equation}\label{eq:graded-hom}
 \Hom_{\mathrm{gr}}(M_X\langle s\rangle,M_Y\langle t\rangle)
 =\Hom^{(s-t)}(X,Y).
\end{equation}
\item Every graded indecomposable is a shift of one of the \(M_X\), and
\(\operatorname{grmod}R\) and every interval algebra \(R_m\) are directed.
\end{enumerate}
\end{proposition}
\begin{proof}
(1) The grading extends to finite direct sums, hence to
\(\End_R(\bigoplus_i e_iR)=R\), where \(a\in R\) acts on \(R_R\) by left
multiplication. Evaluation at 1 identifies \(\Hom_R(R,X)\) with \(X\), with
right action \((f\cdot a)(u)=f(au)\) for \(u\in R\).

(2) A module map \(M_X\langle s\rangle\to M_Y\langle t\rangle\) is
postcomposition by a map \(X\to Y\). Decomposing it into homogeneous components
proves \eqref{eq:graded-hom}.

(3) Proposition~\ref{prop:interval-modules}
gives the classification up to shift. Each graded indecomposable has
endomorphism ring \(k\), and a nonzero map between distinct graded
indecomposables strictly decreases the shift. There can be no cycle.
Fullness of the interval subcategory proves the last assertion.
\end{proof}

\subsection{Standard forms and the lower bound}\label{subsec:graded}
Proposition~\ref{prop:graded-classification} requires a positive Hom
grading. For a basic representation-finite algebra \(A\), we obtain one
on a standard form of \(A\): an algebra with the same Auslander--Reiten
quiver, whose category of indecomposable modules is the mesh category of
that quiver (Definition~\ref{def:mesh}).

\begin{theorem}[{\cite[Sections~2.4, 2.8, 2.9, and~5.1]{BG};
\cite[Section~3.1]{BrG}}]\label{thm:standard}
Let \(A\) be a basic representation-finite algebra. There is a basic
representation-finite algebra \(S\) with the following properties.
\begin{enumerate}
\item The Auslander--Reiten quivers of \(A\) and \(S\)
are isomorphic, preserving arrow multiplicities, translation, and
projective and injective vertices.
\item For a suitable mesh presentation of this quiver, the category of
indecomposable \(S\)-modules is \(k\)-linearly equivalent to its mesh
category, identifying vertices
with the corresponding modules.
\end{enumerate}
\end{theorem}
We call any such algebra \(S\) a \emph{standard form} of \(A\).
We use path length in a mesh presentation to grade the Hom spaces.
\begin{construction}\label{con:hom-grading}
For a basic representation-finite algebra \(A\), choose a standard form
\(S\), and a mesh presentation and an equivalence as in
Theorem~\ref{thm:standard}.
For \(X,Y\in\ind S\), let \(\Hom^{(d)}(X,Y)\) be the image under this
equivalence of the subspace spanned by paths of length \(d\).
Put \(\Hom^{(d)}(X,Y)=0\) for \(d<0\).
\end{construction}

The following lemma records the properties of these subspaces and the
numerical comparison with \(A\); cf. \cite[Section~5.1]{BG}.
\begin{lemma}\label{lem:standard-grading}
\begin{enumerate}
\item The subspaces in Construction~\ref{con:hom-grading} give a grading
\begin{equation}\label{eq:hom-grading}
 \Hom_S(X,Y)=\bigoplus_{d\geq0}\Hom^{(d)}(X,Y),\qquad
 \Hom^{(0)}(X,Y)=
 \begin{cases}k1_X&X=Y,\\0&X\ne Y,\end{cases}
\end{equation}
with composition adding degrees. The part of positive degree is the radical,
and the components of degree at least two give its square. In particular,
\(\Hom^{(1)}(X,Y)\cong\Irr_S(X,Y)\).
\item One has \(|S|=|A|\) and
\begin{equation}\label{eq:standard-surplus}
 \sigma(S)=\sigma(A).
\end{equation}
\end{enumerate}
\end{lemma}
\begin{proof}
(1) The mesh relations are homogeneous of degree two, so the path spaces
give the grading \eqref{eq:hom-grading}. The part of positive degree is an
ideal with quotient \(\prod_X k1_X\), which is semisimple. It is
nilpotent: there are finitely many finite-dimensional Hom spaces, so their
degrees are bounded, and a composite of more maps of positive degree than
this bound vanishes. It is therefore the radical. Every path of length at least two is a product of arrows, so
those components give the radical square and the components of degree one
give the irreducible spaces.

(2) Theorem~\ref{thm:standard}(1) preserves the number of indecomposables,
arrows, and indecomposable projectives; in particular \(|S|=|A|\). Lemma~\ref{lem:count} then gives
\eqref{eq:standard-surplus}.
\end{proof}

By Lemma~\ref{lem:standard-grading}(1), the grading \eqref{eq:hom-grading}
is a positive Hom grading in the sense of
Definition~\ref{def:positive-hom-grading}, so
Proposition~\ref{prop:graded-classification} applies to \(R=S\).
\begin{setting}\label{set:grading}
For the remainder of this section and Section~\ref{sec:equality}, choose a standard
form \(S\) of a basic representation-finite algebra \(A\), and fix the grading
\eqref{eq:hom-grading}. Proposition~\ref{prop:graded-classification} gives
\(S\) a nonnegative grading and the graded lifts \(M_X\) of its part~(1).
By Lemma~\ref{lem:standard-grading}(1), \(S^{(0)}=\bigoplus_i ke_i\) for the
primitive idempotents \(e_i\).
There are finitely many finite-dimensional Hom spaces, so we may choose
an integer \(b\geq1\) bounding all nonzero degrees in
\eqref{eq:hom-grading}. In particular \(S^{(d)}=0\)
for \(d>b\).
\end{setting}

By Propositions~\ref{prop:interval-modules}
and~\ref{prop:graded-classification}, all interval algebras \(S_m\) are
representation-finite and representation-directed.

For the dual numbers we can describe the grading and all the interval
algebras explicitly.
\begin{example}\label{ex:dual-numbers}
Let \(S=k[t]/(t^2)\). Its indecomposables are \(P=S\) and \(U=S/(t)\),
with irreducible maps the quotient \(\pi:P\to U\) and the socle inclusion
\(\iota:U\to P\). The cycle \(P\to U\to P\) shows that \(S\) is
not representation-directed, and Lemma~\ref{lem:count} gives
\(\sigma(S)=2(2-1)-2=0\).
Every morphism is a linear combination of identities,
\(\pi\), \(\iota\), and \(\iota\pi\), subject to the mesh relation
\(\pi\iota=0\). Thus the indecomposable category is its mesh category,
and \(S\) may be taken as its own standard form. Since \(\iota\pi\) is
multiplication by \(t\), path length gives \(t\) degree two. Accordingly, \(M_P\) has one-dimensional components in degrees
zero and two, while \(M_U\) is concentrated in degree one.

For \(m\geq2\), let \(Q_m\) be the quiver with vertices
\(0,\ldots,m\) and arrows \(r\to r+2\) for \(0\leq r\leq m-2\). Then
\(S_m\) is the opposite of \(kQ_m\) modulo all paths of length two, so
right \(S_m\)-modules are representations of \(Q_m\), with arrows
\(r\to r+2\), in which all composites of two arrows vanish. Its indecomposables are the \(m+1\)
simples and the \(m-1\) shifts of \(M_P\) supported in degrees \(r\) and
\(r+2\), for \(0\leq r\leq m-2\). Each nonsimple indecomposable has an irreducible map
from its socle and one to its top. Hence
\[
 |S_m|=m+1,\qquad |\ind S_m|=2m,\qquad a(S_m)=2(m-1),
 \qquad\sigma(S_m)=0.
\]
\end{example}
\begin{theorem}\label{thm:lower-bound}
Every representation-finite algebra \(A\) satisfies
\(\chi(\mod A)\geq|A|\).
\end{theorem}
\begin{proof}
By Lemma~\ref{lem:invariance}, we may assume that \(A\) is basic.
Choose a standard form \(S\).
Since the interval algebras \(S_m\) are representation-finite and
representation-directed, Corollary~\ref{cor:directed} gives
\(\sigma(S_m)\geq0\) for every \(m\). Divide
\eqref{eq:interval-surplus} by \(m+1\) and let \(m\to\infty\).
This gives \(\sigma(S)\geq0\), and \eqref{eq:standard-surplus} gives
\(\sigma(A)\geq0\).
\end{proof}

\section{Equality in the lower bound and special biserial algebras}\label{sec:equality}
Theorem~\ref{thm:lower-bound} proves \(\chi(\mod A)\geq|A|\), but its
asymptotic argument does not by itself determine when equality holds. We first bound, on each finite interval, a nonnegative invariant of
interval algebras that does not increase under primitive deletion. We then
show that the almost split sequences of a graded algebra occur in long
interval algebras.

\subsection{Monotone invariants of interval algebras}\label{subsec:interval-equality}
The following lemma applies to any invariant with properties (a)--(c)
below; we use it for the
surplus in Proposition~\ref{prop:interval-equality}. For an interval
algebra \(R_m\), call the algebras \(R_m/R_m\eta R_m\), for
idempotents \(\eta\) of the diagonal subalgebra
\((R^{(0)})^{m+1}\), its \emph{idempotent quotients}.
\begin{lemma}\label{lem:separated-intervals}
Let \(R\) be a nonnegatively graded finite-dimensional algebra with
\(R^{(d)}=0\) for \(d>b\). Let \(f\) be a real-valued invariant defined
on the idempotent quotients of the interval algebras of \(R\) and on
finite products of these algebras, and let \(\lambda\) be a real number,
with the following properties.
\begin{enumerate}[label={\upshape(\alph*)}]
\item \(f\) is additive on products.
\item For every idempotent quotient \(C\) of an interval algebra and every
primitive idempotent \(\varepsilon\) of the image in \(C\) of the diagonal
subalgebra, \(f(C/C\varepsilon C)\leq f(C)\).
\item \(f(R_m)\geq0\) for every \(m\), and
\(f(R_m)=(m+1)\lambda+o(m)\) as \(m\to\infty\).
\end{enumerate}
Then, for every \(n\geq0\),
\[
 0\leq f(R_n)\leq(n+b+1)\lambda.
\]
In particular, \(\lambda=0\) forces \(f(R_n)=0\) for every \(n\).
\end{lemma}
\begin{proof}
Fix \(n\), put \(\ell=n+b+1\), and for \(q\geq1\) put
\(M=q\ell-b-1\). Consider the \(q\) intervals
\[
 [j\ell,j\ell+n],\qquad 0\leq j<q,
\]
inside \([0,M]\), and delete the degrees in the gaps between them, that is, pass to
the idempotent quotient by the sum of the diagonal idempotents of these
degrees.
Distinct intervals among them are separated by \(b\) degrees, so all
blocks between them vanish. No map within one of these intervals factors
through a deleted degree, since degrees are monotone under multiplication.
The quotient is therefore \(\prod_{j=0}^{q-1}R_n\).
Deleting the corresponding primitive idempotents of the diagonal
subalgebra one at a time passes through idempotent quotients of \(R_M\).
Monotonicity and additivity therefore give
\[
 qf(R_n)\leq f(R_M)=(q\ell-b)\lambda+o(q).
\]
Divide by \(q\) and let \(q\to\infty\). Nonnegativity gives the other bound.
\end{proof}

We apply this lemma with \(f=\sigma\) to the standard form \(S\) and
the grading in Setting~\ref{set:grading}.
\begin{proposition}\label{prop:interval-equality}
If \(\sigma(S)=0\), then \(\sigma(S_n)=0\) and \(S_n\) is biserial for
every \(n\geq0\).
\end{proposition}
\begin{proof}
All the interval algebras are basic, since their diagonal quotients are
products of copies of \(S^{(0)}=\bigoplus_i ke_i\). They are
representation-finite and representation-directed by
Propositions~\ref{prop:interval-modules}
and~\ref{prop:graded-classification},
and so are their idempotent quotients, by repeated application of
Lemma~\ref{lem:quotient-closure}; a primitive deletion of a basic
algebra is basic.
The hypotheses of Lemma~\ref{lem:separated-intervals} hold with
\(\lambda=\sigma(S)\): (a) by Lemma~\ref{lem:invariance},
(b) by Proposition~\ref{thm:directed-deletion}, and (c) by
Corollary~\ref{cor:directed} and Theorem~\ref{prop:interval-count}(2).
That lemma gives \(0\leq\sigma(S_n)\leq(n+b+1)\sigma(S)=0\), and
Corollary~\ref{cor:directed} shows that \(S_n\) is biserial.
\end{proof}

\subsection{Almost split sequences of long interval algebras}\label{subsec:ar-equality}
We first define the number \(\beta(C)\) that appears in the
characterization of biserial algebras of Theorem~\ref{thm:biserial-ar},
then prove that the almost split sequences of a graded algebra occur in
finite intervals.
\begin{definition}\label{def:beta}
For a representation-finite algebra \(C\), let \(\beta(C)\) be the
maximum number of nonprojective indecomposable summands in the middle
term of an almost split sequence of right \(C\)-modules, counting
multiplicities. An empty maximum is zero.
\end{definition}

\begin{theorem}[{\cite[Theorem~4.7]{ARuni}; \cite[p.~172 and Corollary to Theorem~1, p.~173]{SW}}]\label{thm:biserial-ar}
A representation-finite algebra \(C\) is biserial if and only if
\(\beta(C)\leq2\).
\end{theorem}
The implication that \(C\) is biserial if \(\beta(C)\leq2\) is due to Auslander
and Reiten \cite[Theorem~4.7]{ARuni}, and Skowro\'nski and Waschb\"usch
prove the converse \cite[p.~173]{SW}. In \cite{SW} the equivalence is
stated with the bound on both left and right modules, that is, on
\(\beta(C)\) and \(\beta(C^\op)\). By an observation of Gabriel recorded
there \cite[p.~172]{SW}, \(\beta(C)=\beta(C^\op)\) whenever one of them is
at least three, so the bound on \(\beta(C)\) suffices.

\begin{proposition}\label{prop:ar-transfer}
Let \(R\) be a nonnegatively graded representation-finite algebra.
For all sufficiently large \(m\),
\[
 \beta(R)\leq\beta(R_m).
\]
\end{proposition}
\begin{proof}
By \cite[Theorem~3.5]{GG2}, the category \(\operatorname{grmod}R\) has
almost split sequences, and forgetting the grading sends them to almost
split sequences of \(\mod R\). By Proposition~\ref{prop:interval-modules},
every indecomposable \(R\)-module has a graded lift, so every almost split
sequence of \(\mod R\) arises in this way. A graded almost split sequence
whose terms are supported in \([0,m]\) remains almost split in the full
subcategory of graded modules supported in \([0,m]\), that is, in
\(\mod R_m\); by \cite[Lemma~3.6]{GG2}, each graded almost split sequence
has such a shift for \(m\) large.

For each nonprojective middle summand, take the graded almost split
sequence ending at that particular shift of the summand. It gives a
nonsplit epimorphism onto the summand, so the summand is nonprojective in
every interval containing the support of this sequence. There are only
finitely many almost split sequences and summands to consider. Shift each
almost split sequence of \(R\), together with the sequences ending at its
nonprojective middle summands, by the same integer so that the least degree
occurring in this family is zero, and choose
\(m_0\) such that \([0,m_0]\) contains all their supports. For every \(m\geq m_0\), every nonprojective middle summand
counted over \(R\) is counted over \(R_m\) with the same multiplicity.
Taking the maximum proves the inequality.
\end{proof}

\subsection{Algebras attaining the lower bound}
We now show that an algebra with zero surplus is special biserial, by passing from
the interval algebras of a standard form back to the algebra.

\begin{proposition}\label{prop:equality-forward}
If \(A\) is basic and representation-finite and \(\sigma(A)=0\), then
\(A\) is special biserial.
\end{proposition}
\begin{proof}
Choose a standard form \(S\) of \(A\). Then \(\sigma(S)=0\) by
\eqref{eq:standard-surplus}. By Proposition~\ref{prop:interval-equality},
\(C=S_m\) is biserial for every \(m\), so Theorem~\ref{thm:biserial-ar} gives
\(\beta(C)\leq2\). Choose \(m\) sufficiently large for Proposition~\ref{prop:ar-transfer}; it gives the same
bound for \(S\). The isomorphism \(\Gamma_A\cong\Gamma_S\) in
Theorem~\ref{thm:standard} preserves projective vertices and arrow
multiplicities, so \(\beta(A)=\beta(S)\leq2\). Another application
of Theorem~\ref{thm:biserial-ar} shows that \(A\) is biserial.
It is therefore special biserial by Theorem~\ref{thm:biserial-special}.
\end{proof}

\begin{proposition}[{\cite[Theorem~3.5]{BHK}}]\label{prop:equality-converse}
If \(A\) is representation-finite and special biserial, then
\(\chi(\mod A)=|A|\).
\end{proposition}

Theorem~\ref{thm:lower-bound} and
Propositions~\ref{prop:equality-forward} and~\ref{prop:equality-converse}
prove \cite[Conjecture~A]{BHK} over an algebraically closed field.
\begin{theorem}\label{thm:main}
Let \(A\) be a finite-dimensional representation-finite algebra over an
algebraically closed field. Then the following hold.
\begin{enumerate}
\item \(\chi(\mod A)\geq |A|\).
\item \(\chi(\mod A)=|A|\) if and only if \(A\) is special biserial.
\end{enumerate}
\end{theorem}
\begin{proof}
Magnitude, \(|A|\), and being special biserial are invariant under
Morita equivalence. We may therefore assume that \(A\)
is basic.
\begin{enumerate}
\item Theorem~\ref{thm:lower-bound} proves the inequality.
\item The implication \(\Rightarrow\) is
Proposition~\ref{prop:equality-forward}, and \(\Leftarrow\) is
Proposition~\ref{prop:equality-converse}.
\end{enumerate}
\end{proof}

\appendix
\section{Direct proofs for ideal quotients of representation-directed algebras}\label{app:structural}
We work in Setting~\ref{set:directed}. This appendix proves directly the
three results listed at the start of Section~\ref{sec:deletion}.
Subsection~\ref{app:multiplicity} proves Proposition~\ref{prop:boundary-one}.
Subsection~\ref{app:heights} determines the arrows at \(\tau\)-projectives
and \(\tau\)-injectives and constructs a function on the indecomposables
of \(\cD\) as in Lemma~\ref{lem:height-arrows}, which
Subsection~\ref{app:inequality} uses.
Subsection~\ref{app:poset} proves Proposition~\ref{prop:poset-equivalence},
and Subsection~\ref{app:inequality} proves Corollary~\ref{prop:intrinsic}.
The
proofs use the evaluation functor
\(F(X)=Xe\), the \(\tau\)-sequences of
Proposition~\ref{prop:quotient-sequences}, and Auslander--Reiten duality,
and do not use \cite{RV}.

\subsection{A bound on composition multiplicities}\label{app:multiplicity}
By \eqref{eq:evaluation}, \(d_Z=\dim_k\Hom_A(Z,I)\). We therefore
first bound \(\dim_k\Hom_A(Z,I)\) when \(\Ext_A^1(U,Z)=0\) for every
submodule \(U\subseteq I\).
The idea of counting isomorphism classes of kernels is taken from the
argument in \cite[Section~6]{RV}. Under this vanishing hypothesis it
gives the following lemma.

\begin{lemma}\label{lem:finite-kernel}
Let \(A\) be representation-finite and \(I_0\) an injective \(A\)-module.
Suppose \(Z\in\mod A\) satisfies \(\End_A(Z)=\End_A(I_0)=k\) and
\(\Ext_A^1(U,Z)=0\) for every submodule \(U\subseteq I_0\).
Then \(\dim_k\Hom_A(Z,I_0)\leq1\).
\end{lemma}
\begin{proof}
Let \(f,g:Z\to I_0\) be nonzero maps with isomorphic kernels.
If the kernels are nonzero, apply \(\Hom_A(-,Z)\) to
\(0\to\ker f\to Z\to\operatorname{Im}f\to0\).
Since \(\Ext_A^1(\operatorname{Im}f,Z)=0\), an isomorphism
\(\ker f\to\ker g\), followed by inclusion in \(Z\), extends
to an endomorphism of \(Z\).
The extension is a nonzero scalar, so the two kernels are equal as
submodules. They are also equal if both are zero. Then
\(\operatorname{Im}f\cong Z/\ker f=Z/\ker g\cong\operatorname{Im}g\), and
this isomorphism extends to a nonzero scalar
endomorphism of \(I_0\) by injectivity. Thus \(g\) is proportional to \(f\).

Different lines in \(\Hom_A(Z,I_0)\) therefore have nonisomorphic kernels.
There are only finitely many kernel isomorphism classes: their lengths are
bounded by the length of \(Z\), and there are finitely many indecomposables.
Since \(k\) is algebraically closed, it is infinite. A space of dimension
at least two over \(k\) has infinitely many
lines, a contradiction.
\end{proof}

At the \(\tau\)-injectives of \(\cD\), Auslander--Reiten duality gives
the vanishing of \(\Ext_A^1(U,Z)\) required in
Lemma~\ref{lem:finite-kernel}; the \(\tau\)-projectives follow by duality.

\begin{proof}[Direct proof of Proposition~\ref{prop:boundary-one}]
Let \(Z\) be an indecomposable \(A\)-module. If \(Z\) is injective, or if \(\tau_A^-Z\) belongs to \(\mod B\),
then \(\Ext_A^1(U,Z)=0\) for every \(U\subseteq I\).
In the second case, apply Auslander--Reiten duality
\eqref{eq:ar-duality}; ordinary Hom already vanishes:
\[
 \Hom_A(\tau_A^-Z,U)\hookrightarrow\Hom_A(\tau_A^-Z,I)
 \cong\Dk((\tau_A^-Z)e)=0.
\]
Since \(I\) is the indecomposable injective envelope of \(E\),
Lemma~\ref{lem:evaluation}(1) gives \(\End_A(Z)=\End_A(I)=k\).
Lemma~\ref{lem:finite-kernel} therefore
applies and gives \(d_Z=\dim_k\Hom_A(Z,I)\leq1\).
The same bound holds when \(Z\) is projective or \((\tau_A Z)e=0\).
To see this, apply the argument to \(A^\op\), the same idempotent \(e\), and
the right \(A^\op\)-module \(\Dk Z\). The duality \(\Dk\) satisfies
\((\Dk Z)e\cong\Dk(Ze)\), so it preserves dimension at \(e\) and sends
\(\mod B\) to \(\mod B^\op\). It preserves directedness, exchanges
projectives with injectives, and satisfies
\(\Dk\tau_A Z\cong\tau_{A^\op}^-\Dk Z\). Hence \(Z\) is projective or has
\((\tau_A Z)e=0\) exactly when \(\Dk Z\) is injective or has
\((\tau_{A^\op}^-\Dk Z)e=0\).

(1) By Lemma~\ref{rem:iyama-boundary}, every \(\tau\)-injective of \(\cD\)
is either injective over \(A\) or has its inverse translate in \(\mod B\),
and dually for \(\tau\)-projectives. The bound \(d_Z\leq1\) for such \(Z\),
and its dual, therefore give
\(d_X\leq1\), and \(d_X\geq1\) for objects of \(\cD\) by
Lemma~\ref{lem:evaluation}(1).

(2) Suppose \((\tau_A Y)e=0\). If \(Ye\ne0\), then \(Y\) is
\(\tau\)-projective in \(\cD\) by Lemma~\ref{rem:iyama-boundary}, so (1)
gives \(d_Y=1\); if \(Ye=0\), then \(d_Y=0\). The second implication is
dual.
\end{proof}

\subsection{Arrows at \texorpdfstring{\(\tau\)}{tau}-projectives and the height function}\label{app:heights}
The \(\tau\)-sequences of \(\cD\) give the following identities. Ringel and
Vossieck use the equations in part~(2) \cite[Section~1, p.~217]{RV}.
\begin{lemma}\label{lem:quotient-identities}
In Setting~\ref{set:directed}, the following hold.
\begin{enumerate}
\item If \(\tau_{\cD}^+X\ne0\), then for every \(Y\in\ind\cD\),
\begin{equation}\label{eq:paired-middle}
 \dim_k\Irr_{\cD}(Y,X)=\dim_k\Irr_{\cD}(\tau_{\cD}^+X,Y).
\end{equation}
\item For \(X\in\ind\cD\), evaluation at \(e\) satisfies
\begin{equation}\label{eq:dimension-defects}
 d_X-d_{\theta_{\cD}^+X}+d_{\tau_{\cD}^+X}=\indicator_{\{X=P\}},\qquad
 d_X-d_{\theta_{\cD}^-X}+d_{\tau_{\cD}^-X}=\indicator_{\{X=I\}}.
\end{equation}
Here \(\indicator_{\{X=P\}}\) is one if \(X\cong P\) and zero
otherwise, and similarly for \(I\).
\end{enumerate}
\end{lemma}
\begin{proof}
(1) By Proposition~\ref{rem:tau-properties}(2), the right \(\tau\)-sequence at
\(X\) is also the left \(\tau\)-sequence at \(\tau_{\cD}^+X\), so
\(\theta_{\cD}^-(\tau_{\cD}^+X)\cong\theta_{\cD}^+X\).
Lemma~\ref{lem:tau-multiplicity}, applied to both sequences, gives
\eqref{eq:paired-middle}.

(2) The category \(\cD\) is a strict \(\tau\)-category by
Proposition~\ref{prop:quotient-sequences}. Evaluate the first exact
sequence of Lemma~\ref{lem:strict-resolution} at \(P\). Since
\(\cD(P,Z)\cong Ze\) by \eqref{eq:evaluation}, this gives
\(d_X-d_{\theta_{\cD}^+X}+d_{\tau_{\cD}^+X}=\dim_k\mathsf S_X(P)\). By
Lemma~\ref{lem:evaluation}(1), \(\mathsf S_X(P)\) is \(k\) if \(X\cong P\) and zero otherwise. The
second identity in \eqref{eq:dimension-defects} follows in the same way,
by evaluating the second sequence of Lemma~\ref{lem:strict-resolution} at
\(I\) and using \(\cD(Z,I)\cong\Dk(Ze)\).
\end{proof}

With Proposition~\ref{prop:boundary-one}(1), proved in
Subsection~\ref{app:multiplicity}, these identities determine the arrows
at \(\tau\)-projectives and \(\tau\)-injectives. Ringel and Vossieck also
observe this \cite[Section~1, p.~218]{RV}.
\begin{corollary}\label{cor:boundary-arrows}
In \(\Gamma_{\cD}\), each \(\tau\)-projective other than \(P\) has exactly one incoming
arrow, with multiplicity one; \(P\) has none. Dually, each
\(\tau\)-injective other than \(I\) has exactly one outgoing arrow,
and \(I\) has none.
\end{corollary}
\begin{proof}
For a \(\tau\)-projective \(X\ne P\), equation
\eqref{eq:dimension-defects} and Proposition~\ref{prop:boundary-one}(1)
give \(d_{\theta_{\cD}^+X}=d_X=1\).
Every indecomposable summand of \(\theta_{\cD}^+X\) has positive integral
dimension at \(e\), so there is exactly one, occurring once. Lemma~\ref{lem:tau-multiplicity}
then shows that \(X\) has exactly one incoming arrow, of multiplicity one. For \(P\), Lemma~\ref{lem:evaluation}(1),(2)
gives \(d_P=\dim_k\End_A(P)=1\), so equation
\eqref{eq:dimension-defects} gives \(d_{\theta_{\cD}^+P}=0\). Thus there are no
incoming arrows. The second formula in \eqref{eq:dimension-defects}
gives the outgoing counts by the same argument.
\end{proof}

The following proposition constructs on \(\Gamma_{\cD}\) a function as
in Lemma~\ref{lem:height-arrows}(2), and describes the nonzero maps of
\(\cD\) by powers of the radical; Subsection~\ref{app:inequality} uses
both. Its proof uses Corollary~\ref{cor:boundary-arrows} and the arrow
pairing \eqref{eq:paired-middle}.
\begin{proposition}\label{prop:height}
There is a unique function \(h:\ind\cD\to\mathbb Z_{\geq0}\), called
the \emph{height}, with \(h(P)=0\)
such that every arrow \(X\to Y\) satisfies \(h(Y)=h(X)+1\), and
\(h(X)=h(\tau_{\cD}^+X)+2\) whenever \(\tau_{\cD}^+X\ne0\).
For distinct \(X,Y\) with \(\cD(X,Y)\ne0\), put \(\ell=h(Y)-h(X)\).
Then \(\ell>0\) and
\begin{equation}\label{eq:height-radical}
 \cD(X,Y)=\rad_{\cD}^{\ell}(X,Y),\qquad
 \rad_{\cD}^{\ell+1}(X,Y)=0.
\end{equation}
Put \(L=h(I)\). All heights lie in \([0,L]\), with \(P,I\) the unique
objects of heights \(0,L\), respectively.
\end{proposition}

\begin{proof}
Order the vertices so that the source of every arrow precedes its target;
this is possible by directedness. We define the heights in that order. At \(P\),
set \(h(P)=0\).
For a \(\tau\)-projective \(Q\ne P\), its unique incoming arrow
\(Y\to Q\) gives \(h(Q)=h(Y)+1\).
If \(X\) is not \(\tau\)-projective, then for each incoming arrow
\(Y\to X\), equation~\eqref{eq:paired-middle} gives
\(\Irr_{\cD}(\tau_{\cD}^+X,Y)\ne0\), so there is an arrow \(\tau_{\cD}^+X\to Y\).
Both \(\tau_{\cD}^+X\) and \(Y\) precede \(X\), so the arrow
\(\tau_{\cD}^+X\to Y\) gives \(h(Y)=h(\tau_{\cD}^+X)+1\) by induction. Set \(h(X)=h(\tau_{\cD}^+X)+2\). Induction proves the asserted arrow
and translation identities.

Choose maps representing bases of the irreducible spaces. By Lemma~\ref{lem:path-expansion}, every nonisomorphism is a linear
combination of composites of these maps. Every path from \(X\)
to \(Y\) has length \(h(Y)-h(X)\), so every map \(X\to Y\) lies in
\(\rad_{\cD}^{h(Y)-h(X)}\) and no nonzero map lies in the next power; this
proves \eqref{eq:height-radical}. The nonzero maps \(P\to X\) and \(X\to I\)
in \eqref{eq:evaluation} place all heights in the stated interval and give
uniqueness at its endpoints.
There is a path from \(P\) to every indecomposable. Since \(h(P)=0\)
and height increases by one along each arrow, the height of each
indecomposable is the length of any path from \(P\) to it. This proves
uniqueness of \(h\).
\end{proof}

\subsection{The equivalence with poset representations}\label{app:poset}
Construction~\ref{con:poset} describes the subspaces of \(Xe\) using
maps from the \(\tau\)-projectives. We prove fullness by induction and
realize an arbitrary poset representation by taking a quotient of an
explicit \(A\)-module.

\begin{proof}[Direct proof of Proposition~\ref{prop:poset-equivalence}]
We divide the proof into three steps.

\textbf{Step 1: The functor and its values on \(\tau\)-projectives.}
Since \(F\) is faithful and \(F(Q_t),F(Q_s)\) are one-dimensional,
\(\dim_k\cD(Q_t,Q_s)\leq1\).
Composites of nonzero maps between these objects evaluate to composites of
isomorphisms between lines, so they are nonzero. Thus the relation is
transitive; identity maps give reflexivity, and directedness gives
antisymmetry. Hence \(T\) is a poset.
For \(s\leq t\), normalize \(w:Q_t\to Q_s\) so that \(wu_t=u_s\).
Then \(fu_s=(fw)u_t\), proving \(F_s(X)\subseteq F_t(X)\).
Postcomposition by \(a:X\to Y\) sends \(fu_t\) to \((af)u_t\), so
\(G\) is a functor and is faithful.

Directedness gives \(\cD(Q_t,P)=0\). Also \(F_s(Q_t)=k\) exactly
when \(t\leq s\), giving the stated values on \(P,Q_t\).
Finally, \(F(I)=Ie\cong\Dk(eAe)=k\) and \(\cD(Q_t,I)\cong\Dk(Q_te)\ne0\), so
\(G(I)=K_T\).
A map from \(K_{\{s:t\leq s\}}\) to a \(T\)-space is determined by
an arbitrary vector in its \(t\)-subspace. The map
\(\cD(Q_t,X)\to F_t(X)\) is injective because \(F(u_t)\) is an
isomorphism and \(F\) is faithful, and it is surjective by definition.
Maps from \(K_\varnothing\) to \(G(X)\) are arbitrary vectors of
\(F(X)=\cD(P,X)\). Thus every map with source \(G(P)\) or \(G(Q_t)\)
is induced by a map in \(\cD\), proving fullness for these sources.

\textbf{Step 2: Fullness.}
Induct over sources in a topological order. For a non-\(\tau\)-projective
\(X\), use its right \(\tau\)-sequence
\(\tau_{\cD}^+X\xrightarrow{f}\theta_{\cD}^+X\xrightarrow{g}X\).
Every summand of \(\theta_{\cD}^+X\) precedes \(X\). Moreover, \(F(g)\)
is surjective: every map \(P\to X\) is a non-retraction and factors
through \(g\). Given \(a:G(X)\to G(Y)\), induction lifts \(aG(g)\)
to \(b:\theta_{\cD}^+X\to Y\). Faithfulness gives \(bf=0\).
Since \(\tau_{\cD}^+X\ne0\), this is also a left \(\tau\)-sequence,
so \(g\) is a weak cokernel of \(f\). Thus \(b=vg\) for
some \(v:X\to Y\). Surjectivity of \(F(g)\) gives \(G(v)=a\).
This proves fullness on indecomposables, and hence on direct sums.

\textbf{Step 3: Essential surjectivity.}
For each \(\tau\)-projective \(Q\) of \(\cD\) and each \(A\)-module \(V_0\)
with \(V_0=V_0eA\), equivalently a quotient of a finite sum of copies of \(P\),
\begin{equation}\label{eq:generated-ext}
 \Ext_A^1(Q,V_0)=0.
\end{equation}
Indeed this is immediate if \(Q\) is projective over \(A\). Otherwise
\(\tau_A Q\) is annihilated by \(e\), by Lemma~\ref{rem:iyama-boundary}. Since \(V_0=V_0eA\), every map
from \(V_0\) to \(\tau_A Q\) is zero: the generators in \(V_0e\) must map
to \((\tau_A Q)e=0\). Auslander--Reiten duality \eqref{eq:ar-duality} gives
\(\Dk\Ext_A^1(Q,V_0)=\ohom_A(V_0,\tau_A Q)=0\).

Take a \(T\)-space \((W;W_t)\), and form the \(A\)-module
\[
 V=(P\otimes_k W)\oplus\bigoplus_{t\in T}(Q_t\otimes_k W_t).
\]
Here \(A\) acts on the module factor. Our chosen bases identify
\(Ve=W\oplus\bigoplus_tW_t\). Let \(\rho:Ve\to W\) be the sum of
the identity on \(W\) and the inclusions \(W_t\hookrightarrow W\),
and put \(N=\ker\rho\). Define
\[
 \alpha:P\otimes_k N\longrightarrow V,\qquad
 \alpha(ea\otimes x)=xa,
 \qquad V_0=\operatorname{Im}\alpha,\quad M=V/V_0.
\]
The map \(\alpha\) is well defined since \(x\in Ve\). Since \(eAe=k\),
\(\alpha_e\) is the inclusion \(N\hookrightarrow Ve\), and exact
evaluation gives \(V_0e=N\). Thus \(\rho\) identifies \(Me=Ve/N\) with
\(W\), and the quotient map \(Ve\to Me\) with \(\rho\).
Since \(V_0=V_0eA\), we have \(\Ext_A^1(Q_s,V_0)=0\)
by \eqref{eq:generated-ext}. Applying \(\Hom_A(Q_s,-)\) to
\(0\to V_0\to V\to M\to0\) shows that every map \(Q_s\to M\)
lifts to \(V\). Representing morphisms of \(\cD\) by module maps
therefore also gives a surjection \(\cD(Q_s,V)\to\cD(Q_s,M)\).
By Step 1, the \(P\)-summand contributes zero to \(F_s(V)\), and the
\(Q_t\)-summand contributes \(W_t\) exactly when \(t\leq s\). Hence
\[
 F_s(V)=\bigoplus_{t\leq s}W_t\ \subseteq\ Ve,
 \qquad F_s(M)=\rho(F_s(V))=\sum_{t\leq s}W_t=W_s.
\]
The equality \(\sum_{t\leq s}W_t=W_s\) holds because every \(W_t\) in the sum is contained in
\(W_s\), and the term \(t=s\) is present. Thus \(G(M)\) is isomorphic to the given
\(T\)-space \((W;W_t)\).
Steps 1 and 2 prove that \(G\) is fully faithful, and Step 3 proves
essential surjectivity. Thus \(G\) is an equivalence. By its definition,
\(T\) has one element for each \(\tau\)-projective other than \(P\),
so \(|T|=|\ind_1^+\cD|-1\).
\end{proof}

\subsection{Lower bound and equality for the magnitude of the quotient}\label{app:inequality}
By Lemma~\ref{lem:height-arrows}, the height function of
Proposition~\ref{prop:height} determines the number of arrows of \(\cD\),
and hence its magnitude.
\begin{lemma}\label{lem:height-count}
With \(L=h(I)\), one has
\begin{align}
 a(\cD)&=2|\ind\cD|-L-2,\label{eq:height-count}\\
 \chi(\cD)-1&=L-(|\ind_1^+\cD|-1).\label{eq:height-chi}
\end{align}
\end{lemma}
\begin{proof}
The category \(\cD\) is a strict \(\tau\)-category by
Proposition~\ref{prop:quotient-sequences}.
Lemma~\ref{lem:evaluation}(2) and Corollary~\ref{cor:boundary-arrows} give
condition~(1) of Lemma~\ref{lem:height-arrows} with \(X_0=P\) and
\(X_1=I\), and Proposition~\ref{prop:height} gives condition~(2).
Lemma~\ref{lem:height-arrows} now gives both equations.
\end{proof}

The proof of Corollary~\ref{prop:intrinsic} below also uses two general
facts. The first decomposes representations of posets of width at most
two. It is well known; see
\cite[Section~8, Remark, p.~237]{RV} and
\cite[Section~5.2, Corollary]{GR}.
\begin{lemma}\label{lem:two-chains}
Let \(T\) be a finite poset whose antichains have at most two elements.
Then every \(T\)-space is a direct sum of \(T\)-spaces with
one-dimensional total space.
\end{lemma}
\begin{proof}
We recall the usual argument. By Dilworth's theorem \cite{Dilworth}, \(T\)
is a union of at most two chains. The subspaces of a \(T\)-space \(W\)
along these chains, together with \(0\) and \(W\), form two filtrations
\(0=W'_0\subseteq\cdots\subseteq W'_r=W\) and
\(0=W''_0\subseteq\cdots\subseteq W''_q=W\). Choose a complement
\(C_{ij}\) of \((W'_{i-1}\cap W''_j)+(W'_i\cap W''_{j-1})\) in
\(W'_i\cap W''_j\). We show by induction on \(i+j\) that
\(W'_i\cap W''_j=\bigoplus_{a\leq i,\,b\leq j}C_{ab}\). By induction,
\(W'_{i-1}\cap W''_j\), \(W'_i\cap W''_{j-1}\), and their intersection
\(W'_{i-1}\cap W''_{j-1}\) are the direct sums of the \(C_{ab}\) over the
corresponding index sets. Comparing dimensions shows that the sum of the
first two is the direct sum of the \(C_{ab}\) with \(a\leq i\),
\(b\leq j\), and \((a,b)\ne(i,j)\); adding \(C_{ij}\) completes the
induction. Hence bases of the
\(C_{ij}\) form a basis of \(W\) such that every indexed subspace is
spanned by a subset of it. Each basis line, with its intersections with
the indexed subspaces, is a \(T\)-space with one-dimensional total space,
and \(W\) is their direct sum.
\end{proof}

The second shows that a map formed from bases of the irreducible maps
into an object is minimal right almost split. For module categories it
follows from \cite[VII, Lemma~1.2 and Proposition~1.3]{ARS}.
\begin{lemma}\label{lem:basis-almost-split}
Let \(\cC\) be a Hom-finite Krull--Schmidt \(k\)-linear category with
finitely many indecomposable isomorphism classes. Let
\(Y\) be indecomposable and \(g:M\to Y\) a radical map such that, for each
indecomposable \(Z\), the components of \(g\) on the summands of \(M\)
isomorphic to \(Z\) represent a basis of \(\Irr_{\cC}(Z,Y)\). Then \(g\) is
minimal right almost split.
\end{lemma}
\begin{proof}
Evaluating at each indecomposable \(Z\) shows that \(\mathsf P_g\) induces
an isomorphism
\[
 \mathsf P_M/\rad_{\cC}(-,M)\longrightarrow
 \rad_{\cC}(-,Y)/\rad_{\cC}^2(-,Y).
\]
Let \(\mathsf J\subseteq\rad_{\cC}(-,Y)\) be the image of \(\mathsf P_g\). Then
\(\rad_{\cC}(-,Y)=\mathsf J+\rad_{\cC}^2(-,Y)\). Composing with radical maps gives
\(\rad_{\cC}^2(-,Y)\subseteq \mathsf J+\rad_{\cC}^3(-,Y)\), and inductively
\(\rad_{\cC}(-,Y)=\mathsf J+\rad_{\cC}^n(-,Y)\) for every \(n\). The radical is
nilpotent by Lemma~\ref{lem:path-expansion}, so \(\mathsf J=\rad_{\cC}(-,Y)\):
every radical map into \(Y\) factors through \(g\). Since \(g\) is radical,
it is not a retraction, so \(g\) is right almost split. If \(gu=g\) for
some \(u\in\End_{\cC}(M)\), then \(\mathsf P_u\) induces the identity on
\(\mathsf P_M/\rad_{\cC}(-,M)\). Evaluating at \(M\) gives
\(1_M-u\in\rad_{\cC}(M,M)\), so \(u\) is invertible.
\end{proof}

\begin{proof}[Direct proof of Corollary~\ref{prop:intrinsic}]
We divide the proof into four steps.

\textbf{Step 1: \(L\geq|T|\).}
By Lemma~\ref{lem:height-count} and the equality
\(|T|=|\ind_1^+\cD|-1\) of Proposition~\ref{prop:poset-equivalence},
\(\chi(\cD)-1=L-|T|\), so it remains to show \(L\geq|T|\).
Each \(K_U\) is indecomposable, since its endomorphism ring is \(k\).
We use the equivalence \(G\) to regard it as an object of \(\cD\) and
write \(h(K_U)\) for the height of that object.
For upward-closed \(U,V\), the space \(\Hom(K_U,K_V)\) is \(k\) when
\(U\subseteq V\), and zero otherwise. In the nonzero case take the map
that is the identity on total spaces. Every upward-closed subset occurs
in a chain from \(\varnothing\) to \(T\) adding one element at a time.
The inclusions in such a chain are nonzero nonisomorphisms, so each raises
height by at least one by Proposition~\ref{prop:height}. There are
\(|U|\) steps before \(K_U\) and \(|T|-|U|\) after it. Since
\(K_\varnothing=G(P)\) and \(K_T=G(I)\), we obtain
\[
 h(K_U)\geq|U|,\qquad L-h(K_U)\geq|T|-|U|.
\]
Thus \(L\geq|T|\), and hence \(\chi(\cD)\geq1\).
If \(L=|T|\), then
\begin{equation}\label{eq:upset-height}
 h(K_U)=|U|\quad\text{for every upward-closed }U.
\end{equation}
In particular inclusions adding one element give irreducible maps, by
\eqref{eq:height-radical}.

\textbf{Step 2: (1)\(\Rightarrow\)(3).}
Assume (1), that is, \(L=|T|\), so that \eqref{eq:upset-height} holds.
We show that \(T\) has no three-element antichain.
We use the following observation.

\emph{Claim.} If indecomposables \(X,Z_1,Z_2,Y\), with \(Z_1\not\cong Z_2\),
form a commutative square of irreducible maps as below, then
\(X\cong\tau_{\cD}^+Y\).

\[
\begin{tikzcd}
 X\arrow[r,"a"]\arrow[d,"c"']&Z_1\arrow[d,"b"]\\
 Z_2\arrow[r,"d"']&Y
\end{tikzcd}
\qquad ba=dc.
\]
To prove the claim, complete the classes of \(b,d\) to bases of the
incoming irreducible spaces at \(Y\); their sources are nonisomorphic,
so these classes lie in different spaces. By
Lemma~\ref{lem:basis-almost-split}, representatives of these bases form a
minimal right almost split map.
Since minimal right almost split maps ending at \(Y\) are unique up to
isomorphism, and \(\mu_Y^+\) is one of them by
Proposition~\ref{rem:tau-properties}(1), we may take \(b,d\) as components
of \(g=\mu_Y^+:\theta_{\cD}^+Y\to Y\). The nonzero tuple \((a,-c,0,\ldots)\)
is killed by \(g\), so it factors through its weak kernel \(\tau_{\cD}^+Y\).
Hence \(\tau_{\cD}^+Y\ne0\) and \(\cD(X,\tau_{\cD}^+Y)\ne0\).
Both objects have height \(h(Y)-2\). By Proposition~\ref{prop:height},
a nonzero map between distinct indecomposables strictly raises the
height, so
\begin{equation}\label{eq:square-translate}
 X\cong\tau_{\cD}^+Y.
\end{equation}

Suppose \(t_1,t_2,t_3\) are pairwise incomparable in \(T\), and put
\(U=\{s:t_1<s\text{ or }t_2<s\text{ or }t_3<s\}\).
All sets \(U\cup S\), \(S\subseteq\{t_1,t_2,t_3\}\), are upward-closed.
The chosen inclusion maps are identities on total spaces, so their squares
commute. For \(j=2,3\), the four objects
\[
 K_U,\quad K_{U\cup\{t_1\}},\quad K_{U\cup\{t_j\}},\quad
 K_{U\cup\{t_1,t_j\}}
\]
form such a square, with distinct middle objects. The two squares have
common source \(K_U\) and distinct targets.
Equation~\eqref{eq:upset-height} makes their arrows irreducible, and
\eqref{eq:square-translate} gives the same translate \(K_U\) for their
distinct targets. This contradicts the injectivity of \(\tau_{\cD}^+\)
(Proposition~\ref{rem:tau-properties}(3)).

\textbf{Step 3: (3)\(\Rightarrow\)(2).}
If \(T\) has width at most two, Lemma~\ref{lem:two-chains} decomposes
every \(T\)-space into one-dimensional summands, and the equivalence of
Proposition~\ref{prop:poset-equivalence} gives \(d_X=1\) for every
indecomposable of \(\cD\).

\textbf{Step 4: (2)\(\Rightarrow\)(1).}
If all \(d_X=1\), equation~\eqref{eq:dimension-defects}
gives two middle summands at each non-\(\tau\)-projective and one at
each \(\tau\)-projective other than \(P\). Hence
\(a(\cD)=2(|\ind\cD|-|\ind_1^+\cD|)+|\ind_1^+\cD|-1\), and
Proposition~\ref{thm:categorical-magnitude} gives \(\chi(\cD)=1\).
\end{proof}

\end{document}